\documentclass[a4paper]{amsart}

\usepackage{cite}
\usepackage{amsmath}
\usepackage{amssymb}
\usepackage{amsthm}   
\usepackage{empheq}   
\usepackage{graphicx}
\usepackage{color}
\usepackage{appendix}
\usepackage{hyperref} 
\usepackage{mathrsfs}

\newtheorem{theorem}{Theorem}[section]
\newtheorem{lemma}[theorem]{Lemma}

\newtheorem{corollary}[theorem]{Corollary}
\newtheorem{hypothesis}[theorem]{Hypothesis}

\theoremstyle{definition}
\newtheorem{definition}[theorem]{Definition}
\newtheorem{example}[theorem]{Example}

\theoremstyle{remark}
\newtheorem{remark}[theorem]{Remark}

\numberwithin{equation}{section}

\def\cE{{\mathcal{E}}}
\def\cF{{\mathcal{F}}}

\def\cB{{\mathcal{B}}}
\def\cG{{\mathcal{G}}}
\def\cC{{\mathcal{C}}}
\def\cS{{\mathcal{S}}}
\def\cM{\mathcal{M}}
\def\sE{{\mathscr{E}}}
\def\sF{\mathscr{F}}

\def\bP{\mathbf{P}}
\def\bR{\mathbb{R}}
\def\bG{\mathbb{G}}
\def\bQ{\mathbb{Q}}
\def\cN{\mathcal{N}}
\def\cY{\mathcal{Y}}
\def\cW{\mathcal{W}}
\def\cU{\mathcal{U}}
\def\cR{\mathcal{R}}
\def\cH{\mathcal{H}}
\def\fm{\mathfrak{m}}

\begin{document}

\title[Restriction of right process]{On the restriction of a right process outside a negligible set}

\author{ Liping Li}
\address{Fudan University, Shanghai, China.  }
\address{Bielefeld University,  Bielefeld, Germany.}
\email{liliping@fudan.edu.cn}
\thanks{The first named author is a member of LMNS,  Fudan University.  He is also partially supported by NSFC (No.  11931004) and Alexander von Humboldt Foundation in Germany. }

\author{Michael R\"ockner}
\address{Fakult\"at f\"ur Mathematik, Universit\"at Bielefeld,  Postfach 100 131, D-33501 Bielefeld, Germany, and Academy for Mathematics and Systems Science,  CAS,  Beijing}
\email{roeckner@mathematik.unibielefeld.de}





\begin{abstract}

The objective of this paper is to examine the restriction of a right process on a Radon topological space, excluding a negligible set, and investigate whether the restricted object can induce a Markov process with desirable properties. We address this question in three aspects: the induced process necessitates only right continuity; it is a right process, and the semi-Dirichlet form of the induced process is quasi-regular. The main findings characterize the negligible set that meets the requirements within a universally measurable framework. These characterizations can be employed to generate instances of Markov processes that are non-right or (semi-)Dirichlet forms that are non-quasi-regular. Specifically, we will construct an example of a non-tight, strong Feller, symmetric right process on a non-Lusin Radon topological space, whose Dirichlet form is not quasi-regular.
\end{abstract}

\maketitle

 \tableofcontents

\section{Introduction}


In their recent paper \cite{BCR22}, Beznea et al. addressed a frequently posed question: Can a Markov semigroup $(P_t)$ on a Polish space $E$ with the strong Feller property always give rise to a desirable Markov process? Here, the strong Feller property is defined as $P_t(b\cE)\subset bC(E)$, where $b\cE$ represents the collection of all bounded Borel measurable functions on $E$, and $bC(E)$ denotes the family of all bounded continuous functions on $E$. A desirable Markov process refers to a (simple) Markov process exhibiting regular trajectory properties, including right continuity, c\`adl\`ag property, strong Markov property, and more. However, the answer to this question is unequivocally negative. In their work \cite{BCR22}, Beznea et al. devised a compelling set of counterexamples using the restriction method. They initiated with a good Markov process that satisfies the strong Feller property, and subsequently restricted it outside a suitable non-polar set. While this method consistently preserves the strong Feller property, the resulting restricted semigroup cannot induce another desirable (e.g., c\`adl\`ag) Markov process.

The objective of this article is to investigate the impact of restricting a general Markov process outside a negligible set and determine whether the resulting restricted object can still induce a good Markov process. The class of general Markov processes under consideration here refers to \emph{right processes} on a Radon space, which are derived from Meyer's well-known \emph{hypoth\`eses droites} (HD1) and (HD2). According to (HD1), the semigroup $(P_t)$ possesses a right-continuous realization, while (HD2) enables the strong Markov property.  We refer readers to \cite{Ge75, Sh88}, as well as Appendix~\ref{APPA}, for a detailed explanation of these fundamental assumptions.  In the context of this study,  a negligible set refers to a (weakly) $U$-negligible set $N\in \cE^u$, which satisfies $U^\alpha 1_N=0$ outside $N$, where $(U^\alpha)_{\alpha>0}$ represents the resolvent of $(P_t)$, and $\cE^u$ denotes the universally measurable $\sigma$-algebra on $E$. This condition ensures that the restriction of $(U^\alpha)$ outside $N$ remains a family of kernels that satisfy the resolvent equation. We will examine this problem from three distinct levels:
\begin{itemize}
\item[(1)] If $(P_t)$ satisfies only (HD1), can the restricted resolvent induce a transition function that satisfies (HD1)?
\item[(2)] If $(P_t)$ satisfies both (HD1) and (HD2), can the restricted resolvent induce a transition function that satisfies both (HD1) and (HD2)?
\item[(3)] If,  in addition, $(P_t)$ satisfies the \emph{sector condition} and \emph{tightness} with respect to a certain $\sigma$-finite measure, which corresponds to a \emph{quasi-regular semi-Dirichlet form}, is the semi-Dirichlet form restricted outside a set of zero measure still quasi-regular? What is the relationship between this restriction and the restriction of the resolvent?
\end{itemize}

Before addressing these questions, we need to clarify some information regarding topology and measurability. When discussing the restriction of a topological space, we usually use the subspace topology. This is crucial in the examination of the first and third levels. However, the second level is an exception. For right processes, the original topology on the state space only appears in the definition of (HD1), and it is not the intrinsic topology of the right process (the \emph{fine topology} is). As observed by Beznea et al. (see \cite{BCR20}), assigning any so-called \emph{natural topology} to the state space, i.e., topologies coarser than the fine topology, the (Borel) right process is still a (Borel) right process. Therefore, we can also relax the topological restriction when studying the second level problem. Regarding measurability, the universally measurable $\sigma$-algebra has a more complex structure than the Borel $\sigma$-algebra generated by topology. The definition of Radon space, as well as (HD1), is based on this. The importance of universal measurability lies in the fact that most Markov process transformations, such as \emph{killing}, \emph{time change}, and \emph{$h$-transformation}, do not preserve Borel measurability, but do preserve universal measurability. However, the setting of universal measurability also brings some difficulties to the above problem. One of the most prominent difficulties is that we can no longer use the \emph{first hitting/entrance time} to define certain universally measurable exceptional sets because these random times may not be stopping times in general. To overcome this difficulty, we borrow the concepts of \emph{quasi-absorbing} and \emph{quasi-polar} sets from \cite{Sh88} (see Definition~\ref{DEF21}) using the outer measure. This approach, which only requires universal measurability, serves as a generalization of \emph{absorbing} and \emph{polar} sets in the Borel measurable sense. It provides a solution for defining \emph{small} sets in the universally measurable context.

All three levels of the problem will be thoroughly addressed in this article. For the first two levels, our results in Theorems \ref{THM24} and \ref{THM31} establish that the stability holds if and only if the complement of $N$ is quasi-absorbing. This type of restriction can be likened to a Markov chain being confined to one of its irreducible components. Specifically, if $N$ satisfies a slightly stronger condition, namely $U^\alpha 1_N = 0$ on $E$ (not just outside $N$), then a similar characterization reveals that $N$ can only be quasi-polar sets; refer to Theorem~\ref{THM32} for details.
The characterization of the problem at the third level necessitates that $N$ is contained in the complement of a (quasi-)absorbing set that is Borel measurable, and it ensures the consistency of the restrictions on the Dirichlet form and resolvent. At this level, the enhancement of measurability is primarily attributed to the quasi-regularity assumption, which guarantees that the corresponding right process is essentially a Borel right process on a smaller Lusin space. It is important to emphasize that the characterizations of the first and third levels require the assumption that the restricted state space is equipped with the subspace topology. Without this assumption, similar characterizations may not hold, as demonstrated by the two counterexamples provided in Example \ref{EXA46}.

The aforementioned characterizations can be utilized to construct examples of non-right Markov processes or non-quasi-regular (semi-)Dirichlet forms. Beznea et al. presented a simple example in \cite{BCR22}: Consider the restriction of one-dimensional Brownian motion to $\mathbb{R}\setminus\{0\}$ ($\{0\}$ is a $U$-negligible but non-polar set). The restricted semigroup still satisfies the strong Feller property but no longer induces a right-continuous Markov process. Furthermore, the Dirichlet form of Brownian motion restricted to $\mathbb{R}\setminus\{0\}$ is not quasi-regular. In the framework of universal measurability, we can provide even more intriguing examples. Particularly, based on the construction by Salisbury \cite{Sa87}, we can obtain a non-Borel quasi-polar set $Z'\subset \mathbb{R}^n$ ($n\geq 4$) for Brownian motion (refer to Examples~\ref{EXA34} and \ref{EXA45}) and a right process on $\mathbb{R}^n\setminus Z'$, corresponding to the restriction of the Brownian resolvent to $\mathbb{R}^n\setminus Z'$, which is strong Feller, symmetric, but not tight. Specifically, its Dirichlet form is not quasi-regular.

\section{Restriction outside a negligible set under (HD1)}

Let $E$ be a Radon topological space, meaning it is homeomorphic to a universally measurable subset of a compact metric space. In this article, we will use the notations defined in Appendix~\ref{APPA}. Specifically, $\cE$ (resp. $\cE^u$) denotes the Borel $\sigma$-algebra (universally measurable $\sigma$-algebra) on $E$. Let $(P_t)_{t\geq 0}$ be a normal Markov transition function on $(E,\cE^u)$ (refer to Definition~\ref{DEFA1}), and its resolvent is given by
\[
	U^\alpha f=\int_0^\infty e^{-\alpha t}P_t f dt,\quad \alpha\geq 0,  f\in p\cE^u.  
\]
Assume that $(P_t)$ satisfies the first of Meyer's \emph{hypothèses droites} (HD1) as defined in Definition~\ref{DEFA2}. Let us consider the collection
\begin{equation}\label{eq:21}
X=(\Omega, \cF^u, \cF^u_t, X_t, \theta_t, \bP^x),
\end{equation}
which represents the canonical realization of $(P_t)$. In short, $\Omega$ is the family of all right continuous maps from $\bR^+=[0,\infty)$ to $E$, $(X_t)$ is the coordinate process, $(\cF^u, \cF^u_t)$ is the (unaugmented) natural filtration on $\Omega$, and $(X_t)$ satisfies the (simple) Markov property \eqref{eq:A2} with transition semigroup $(P_t)$ under each $\bP^x$. For any finite positive measure $\mu$ on $E$, we set $\bP^\mu:=\int_E \mu(dx)\bP^x$. It is important to note that in this section, $X$ does not necessarily possess the strong Markov property. Generally, a ceremony and a lifetime should be included in the collection \eqref{eq:21}; however, for simplicity, we have omitted them. For further explanation, please refer to \S\ref{APPA8} in the Appendix.

Our objective is to remove a specific negligible subset $N$ from $E$ and investigate whether the restriction of $U^\alpha$ (or $P_t$) to $E\setminus N$ remains associated with a well-behaved Markov process. However, before delving into this, we need to introduce some terminologies related to small sets.

\begin{definition}\label{DEF21}
Let $X$ and $U^\alpha$ be given above.  Let $N\in \cE^u$.  
\begin{itemize}
\item[(1)] $N$ is called \emph{$U$-negligible},  if $U^\alpha 1_N(x)=0$ for one (equivalently,  all) $\alpha>0$ and any $x\in E$;
\item[(2)] $N$ is called \emph{weakly $U$-negligible},  if $U^\alpha 1_N(x)=0$ for one (equivalently,  all) $\alpha>0$ and any $x\in E\setminus N$;
\item[(3)] $E\setminus N$ is called \emph{quasi-absorbing},  if for any initial law $\mu$ carried by $E\setminus N$,  $\{X_t\in E\setminus N, \forall t\geq 0\}$ has full $\bP^\mu$-outer measure; 
\item[(4)] $N$ is called \emph{quasi-polar},  if for every initial law on $E$,  the $\bP^\mu$-outer measure of $\{X_t\notin N:\forall t>0\}$ is equal to $1$.  
\end{itemize}
\end{definition}
\begin{remark}\label{RM22}
\begin{itemize} 
\item[(1)] Here we utilize outer measures because neither $\{X_t\in E\setminus N, \forall t\geq 0\}$ nor $\{X_t\notin N:\forall t>0\}$ necessarily belongs to the augmentation $\cF$ of $\cF^u$ since $N$ is only universally measurable. If the measurability of $N$ is sufficiently good (e.g., $N\in \cE$ or $N$ is \emph{nearly optional} when $X$ further satisfies (HD2)), such that its first hitting/entrance time is $\cF_{t+}$-stopping times ($\cF_t$ represents the augmentation of $\cF^u_t$; see Appendix~\ref{APPA5}), then a quasi-absorbing or quasi-polar set is commonly referred to as an \emph{absorbing} or \emph{polar} set.
\item[(2)]  It is evident that $U$-negligible sets are also weakly $U$-negligible, and the complement of a quasi-polar set is quasi-absorbing. However, the converse statements are not necessarily true. To illustrate this, consider the following example: Let $E=[0,1]\cup \{2\}$, and let $P_t$ be composed of two parts on $E$: Its restriction to $[0,1]$ represents the transition semigroup of reflecting Brownian motion on $[0,1]$, while $P_t(2,\cdot):=\delta_2(\cdot)$ denotes a Dirac measure at $2$. In this case, $\{2\}$ is weakly $U$-negligible and $[0,1]$ is quasi-absorbing. However, $\{2\}$ is neither $U$-negligible nor quasi-polar.
\item[(3)] It is worth noting that the complement of a quasi-absorbing set is weakly $U$-negligible, and a quasi-polar set is $U$-negligible. Conversely, the contraries are not true either. To see this, consider the case where $E\setminus N$ is a quasi-absorbing set. Since ${X_s\in E\setminus N: \forall s\geq 0}\subset {X_t\in E\setminus N}\in \cF^u$ for any fixed $t\geq 0$, it follows that $P_t(x, E\setminus N)=\bP^x(X_t\in E\setminus N)=1$ for any $t\geq 0$ and $x\in E\setminus N$. Hence, $U^\alpha 1_N(x)=0$ for any $x\in E\setminus N$, implying that $N$ is weakly $U$-negligible. Similarly, one can deduce that a quasi-polar set is $U$-negligible.
\item[(4)] In the realm of universal measurability,  the concepts of these sets exhibit a higher level of complexity compared to their counterparts in the Borel measurable sense. To illustrate this, Salisbury \cite{Sa87} constructed a non-Borel universally measurable set $Z\subset \mathbb{R}^n$ ($n\geq 2$) for Brownian motion. Remarkably, this set has the property that all non-constant paths intersect $Z$, indicating that $Z$ is not a quasi-polar set. However, it is noteworthy that every Borel subset of $Z$ is polar. Additionally, Salisbury provided an example of a non-Borel quasi-polar set $Z'$ for $n\geq 4$, such that any Borel set containing $Z'$ is not polar. Further details and explanations can be found in Example~\ref{EXA34}.
\end{itemize}
\end{remark}

For convenience, we use the notation $\bP^\mu_*$ to represent the outer measure of $\bP^\mu$. This function is defined on the entire collection of subsets of $\Omega$ and takes non-negative values. It is important to highlight that the families of quasi-absorbing sets and quasi-polar sets remain unchanged regardless of whether we consider the filtration in \eqref{eq:21} or its augmented version $(\cF,\cF_t)$.

\begin{lemma}\label{LM23}
Let $\mu$ be an initial law on $E$ and $\tilde{\Omega}\subset \Omega$.  Then $\tilde{\Omega}$ has full $(\Omega, \cF^u,\bP^\mu)$-outer measure,  if and only if it has full $(\Omega, \cF,\bP^\mu)$-outer measure.  
\end{lemma}
\begin{proof}
Let us denote the $\bP^\mu$-outer measure relative to $\cF^u$ and $\cF$ by $\bP^\mu_{1*}$ and $\bP^\mu_{2*}$, respectively. For any subset $A \subset \Omega$, we have
\[
	\bP^\mu_{1*}(A)=\inf\{\bP^\mu(B):A\subset B\in \cF^u\},\quad \bP^\mu_{2*}(A)=\inf\{\bP^\mu(B):A\subset B\in \cF\}.  
\]
Let $\cF^\mu$ be the completion of $\cF^u$ relative to $\bP^\mu$ and $\bP^\mu_{3*}$ be the $\bP^\mu$-outer  measure relative to $\cF^\mu$.   Since $\cF^u\subset \cF\subset \cF^\mu$,  it follows that $\bP^\mu_{3*}\leq \bP^\mu_{2*}\leq \bP^\mu_{1*}$.  It suffices to show that $\bP^\mu_{1*}(\tilde{\Omega})=1$ implies $\bP^\mu_{3*}(\tilde{\Omega})=1$.  To do this,  suppose $\bP^\mu_{1*}(\tilde{\Omega})=1$ and take arbitrary $B\in \cF^\mu$ with $\tilde{\Omega}\subset B$.  By means of,  e.g.,  \cite[\S1.4,  Exercise 18b]{Fo99},  one gets a set $B'\in \cF^u$ with $B\subset B'$ and $B'\setminus B\in \cN$ (i.e.  $B'\setminus B\in \cF^\mu$ with $\bP^\mu(B'\setminus B)=0$).  Note that $\bP^\mu_{1*}(\tilde{\Omega})=1$ and $\tilde{\Omega}\subset B\subset B'$ imply $\bP^\mu(B')=1$.  Hence $\bP^\mu(B)=0$,  as arrives at $\bP^\mu_{3*}(\tilde{\Omega})=1$.  That completes the proof.  
\end{proof}

Now we have a position to set up the deletion of a negligible set from $E$.  Let $N\in \cE^u$ and set $\tilde{E}:=E\setminus N$.  The subspace topology on $\tilde{E}$ inherited from $E$ makes $\tilde{E}$ a Radon topological space. The Borel $\sigma$-algebra and the universally measurable $\sigma$-algebra on $\tilde{E}$ are denoted by $\tilde{\cE}$ and $\tilde{\cE}^u$, respectively. It follows that $\tilde{\cE}=\cE|_{\tilde E}$ and $\tilde{\cE}^u=\cE^u|_{\tilde{E}}$.  Suppose that $N$ is weakly $U$-negligible. Then,  we have a well-defined collection of kernels on $(\tilde{E},\tilde{\cE}^u)$ given by:
\[
	\tilde{U}^\alpha \tilde{f}:=(U^\alpha f)|_{\tilde{E}},\quad \alpha>0,  \tilde{f}\in p\tilde{\cE}^u,
\]
where $f\in p\cE^u$ is an arbitrary extension of $\tilde{f}$ to $E$. 
Moreover, it is worth noting that $\alpha \tilde{U}^\alpha 1_{\tilde{E}}\equiv 1$ on $\tilde{E}$ and the restriction $(\tilde{U}^\alpha)$ satisfies the resolvent equation: For $0<\alpha\leq \beta$ and $\tilde{f}\in p\tilde{\cE}^u$, we have:
\[
\tilde U^\alpha \tilde f=\tilde U^\beta \tilde f+(\beta-\alpha) \tilde U^\alpha \tilde U^\beta \tilde f.  
\]
The collection $(\tilde{U}^\alpha)$ is referred to as the \emph{restriction} of $(U^\alpha)$ to $\tilde{E}$.

The presented result expands upon \cite[Theorem2.1]{BCR22} by considering a general case of a (simple) Markov process with right-continuous paths, rather than c\'adl\'ag paths.  A technique involving outer measures is employed, enabling the state space to be a Radon topological space instead of being limited to a Lusin or Polish space.  It is important to note that the definition of (HD1) is closely related to the topology of the state space,  and the utilization of the inherited topology is crucial in the discussion of Theorem\ref{THM24}. An example provided in Example~\ref{EXA46} illustrates that assigning a non-inherited topology to $\tilde{E}$ can lead to a situation where $(\tilde{P}_t)$ satisfies (HD1) even if $\tilde{E}$ is not quasi-absorbing for $X$.

\begin{theorem}\label{THM24}
Let $N\in \cE^u$ be weakly $U$-negligible and $(\tilde{U}^\alpha)_{\alpha>0}$ be the restriction of $(U^\alpha)_{\alpha>0}$ to $\tilde{E}=E\setminus N$,  which is equipped with the subspace topology of $E$.  Then there is a Markov transition function $(\tilde{P}_t)_{t\geq 0}$ on $(\tilde{E},\tilde{\cE}^u)$ satisfying (HD1),  whose resolvent is $(\tilde{U}^\alpha)_{\alpha>0}$,  if and only if $\tilde{E}$ is quasi-absorbing for $X$.    {Meanwhile $(\tilde{P}_t)$ is identical to the restriction of $(P_t)$ to $\tilde{E}$,  i.e.  $\tilde{P}_t(x,B)=P_t(x,B)$ for $t\geq 0,  x\in \tilde{E}$ and $B\in \tilde{\cE}^u$. }
\end{theorem}
\begin{proof}
We first argue the sufficiency and suppose $\tilde{E}$ is quasi-absorbing for $X$.  In view of Remark~\ref{RM22}~(3),  $P_t(x,N)=0$ for any $t\geq 0$ and $x\in \tilde{E}$.  Hence $\tilde{P}_t(x,B):=P_t(x,B)$ for $t\geq 0,  x\in \tilde{E}$ and $B\in \tilde{\cE}^u(\subset \cE^u)$ defines a collection of Markov kernels on $(\tilde{E},\tilde{\cE}^u)$.  It is straightforward to verify that it is a normal transition function on $(\tilde{E},\tilde{\cE}^u)$ whose resolvent is $(\tilde{U}^\alpha)$. We have to prove (HD1) for $(\tilde{P}_t)$.  
To do this,  recall that the collection \eqref{eq:21} is the canonical realization of $(P_t)$.  Set
\[
	\tilde{\Omega}:=\{\omega\in \Omega: X_t\in \tilde{E},\forall t\geq 0\}
\]
and for any ${\omega}\in \tilde{\Omega}$ and $t\geq 0$,  $\tilde{X}_t(\omega):=X_t(\omega)$.  Note that $\theta_t \tilde{\Omega}\subset \tilde{\Omega}$.  Let $\tilde{\theta}_t$ be the restriction of $\theta_t$ to $\tilde{\Omega}$.  Define $\tilde{\cG}$ and $\tilde{\cG}_t$ as the traces of $\cF^u$ and $\cF^u_t$ on $\tilde{\Omega}$,  i.e.  $\tilde{\cG}:=\{B\cap \tilde{\Omega}: B\in \cF^u\}$ and $\tilde{\cG}_t:=\{B\cap \tilde{\Omega}: B\in \cF^u_t\}$.  Clearly $t\mapsto \tilde{X}_t$ is an $\tilde{E}$-valued, right continuous process $\cE^u$-adapted to $\tilde{\cG}_t$.  For any initial law $\mu$ carried by $\tilde{E}$,  let $\tilde{\bP}^\mu$ be the trace of $\bP^\mu$ on $\tilde{\Omega}$,  i.e.  $\tilde{\bP}^\mu(B\cap \tilde{\Omega}):=\inf\{\bP^\mu(B\cap \Gamma): \tilde{\Omega}\subset \Gamma\in \cF^u\}$ for any $B\in \cF^u$.  
By virtue of \cite[(A1.3)]{Sh88} and that $\tilde{E}$ is quasi-absorbing for $X$,  one gets that $\tilde{\bP}^\mu$ is a probability measure on $(\tilde{\Omega}, \tilde{\cG})$.  Applying \cite[(A1.6)]{Sh88} and the Markov property of \eqref{eq:21},  we have that for any $\tilde f\in b\tilde{\cE}^u$ and $\tilde{G}\in \tilde{\cG}_t$,  there exists $f\in b\cE^u$ with $f|_{\tilde{E}}=\tilde f$ and $G\in \cG_t$ with $\tilde{G}=G\cap \tilde{\Omega}$,  so that
\[
\begin{aligned}
	\tilde{\bP}^\mu\{\tilde{f}(\tilde{X}_{t+s})1_{\tilde{G}}\}&=\bP^\mu\{f(X_{t+s})1_G\}=\bP^\mu\{P_s f(X_t)1_G\}\\
	&=\tilde{\bP}^\mu\{ P_s f(X_t)|_{\tilde{\Omega}} 1_{\tilde{G}}\}=\tilde{\bP}^\mu\{ \tilde P_s f(\tilde X_t) 1_{\tilde{G}}\},
\end{aligned}\]
as leads to the Markov property of $\tilde{X}=(\tilde{\Omega},\tilde{\cG},\tilde{\cG}_t,  \tilde{X}_t, \tilde{\bP}^\mu)$ with transition semigroup $\tilde{P}_t$.  Using \cite[(A1.6)]{Sh88} again we also have that for any $A\in \tilde{\cE}$,  $\tilde{\bP}^\mu(\tilde{X}_0\in A)=\bP^\mu(X_0\in A)=\mu(A)$.  Hence the initial law of $\tilde{X}$ is exactly $\mu$.  Eventually (HD1) is verified for $(\tilde{P}_t)$.  

To the contrary,  suppose $(\tilde{P}_t)_{t\geq 0}$ is a transition function on $(\tilde{E},\tilde{\cE}^u)$ satisfying (HD1),  whose resolvent is $(\tilde{U}^\alpha)_{\alpha>0}$.   We first show that $(P_t)$ is an extension of $(\tilde{P}_t)$ in the sense that for any $x\in \tilde{E}$,  $P_t(x,\cdot)$ is carried by $\tilde{E}$ and its restriction to $\tilde{E}$ is equal to $\tilde{P}_t(x,\dot)$.  To do this,  suppose $E$ is embedded into the compact metric space $\hat E$ with the metric $d$ compatible to the topology on $E$,  i.e. the subspace topology of $E$ relative to $(\hat E,d)$ is identical to the original topology of $E$.  Denote by $C_d(\tilde{E})$, and $C_d(E)$ the families of all real $d$-uniformly continuous functions on $\tilde{E}$ and $E$ respectively.  Note that $C_d(\tilde{E})=C_d(E)|_{\tilde{E}}:=\{f|_{\tilde{E}}: f\in C_d(E)\}$ and,  in view of \cite[(A2.1)]{Sh88},  $\cE=\sigma\{C_d(E)\}$ and $\tilde{\cE}=\sigma\{C_d(\tilde{E})\}$.   For any $\tilde{f}=f|_{\tilde{E}}\in C_d(\tilde{E})$ with $f\in C_d(E)$ and $x\in \tilde{E}$,  both $t\mapsto P_t f(x)$ and $t\mapsto \tilde{P}_t\tilde{f}(x)$ are continuous.  Hence $U^\alpha f(x)=\tilde{U}^\alpha \tilde{f}(x)$ for all $\alpha>0$ imply $P_tf(x)=\tilde{P}_t \tilde f(x)$ for all $t\geq 0$.  By a monotone class argument,  it yields that for any $t\geq 0$ and $x\in \tilde{E}$,  $\tilde{P}_t(f\cdot 1_{\tilde{E}})(x)=P_tf(x)$ for any $f\in b\cE$.  Since the extensions of a Borel measure to the universally measurable space are unique (see,  e.g., \cite[(A1.1)]{Sh88}),  it follows that $\tilde{P}_t(f\cdot 1_{\tilde{E}})(x)=P_tf(x)$ for any $f\in b\cE^u$.  Particularly,  $(P_t)$ is an extension of $(\tilde{P}_t)$.  Now let $\tilde{X}=(\tilde\Omega,  \tilde\cF^u,\tilde\cF^u_t,  \tilde X_t,\tilde \theta_t,  \tilde \bP^x)$ be the canonical realization of $(\tilde{P}_t)$ on $\tilde{E}$,  where $\tilde{\Omega}$ is the space of all right continuous maps from $\bR^+$ to $\tilde{E}$.  Applying \cite[Proposition~(19.9)]{Sh88},  one gets that for any initial law $\mu$ carried by $\tilde{E}$,  $\bP^\mu_*(\tilde{\Omega})=1$.  Note that
\[
\tilde{\Omega}=\{\omega\in \Omega: \omega(t)\in \tilde{E},\forall t\geq  0\}=\{X_t\in \tilde{E},\forall t\geq 0\}.  
\]
Therefore $\tilde{E}$ is quasi-absorbing for $X$ by the definition.  
\end{proof}

We readily have the following corollary,  which straightforwardly extends  \cite[Theorem~2.1]{BCR22} to right continuous (simple) Markov processes.  

\begin{corollary}
Let $N\in \cE^u$ be weakly $U$-negligible and $(\tilde{U}^\alpha)_{\alpha>0}$ be the restriction of $(U^\alpha)_{\alpha>0}$ to $\tilde{E}=E\setminus N$,  which is endowed with the subspace topology of $E$.  If there is probability measure $\mu$ on $E$ such that $\mu(N)=0$ and $\{X_t\notin N,\forall t\geq 0\}$ does not have full $\bP^\mu$-outer measure,  then there is no Markov transition function on $(\tilde{E},\tilde{\cE})$ satisfying (HD1),  whose resolvent is $(\tilde{U}^\alpha)$.  
\end{corollary}
\begin{remark}
Suppose $N \in \cE$. Then we have 
\[
	\{X_t \notin N, \forall t \geq 0\} = \left(\{T_N = \infty\} \cap \{X_0 \notin N\}\right) \in \cF,
\]
where $T_N := \inf\{t > 0: X_t \in N\}$ is an $\cF_{t+}$-stopping time. Since $\bP^\mu(X_0 \notin N) = 1$, the condition in this corollary is equivalent to $\bP^\mu(T_N < \infty) > 0$, which corresponds to the condition in \cite[Theorem~2.1]{BCR22}.
\end{remark}

\section{Restriction of a right process outside a negligible set}

The goal of this section is to examine the restriction of a right process outside a negligible set. We consider a normal Markov transition function $(P_t)$ on $(E,\cE^u)$ satisfying (HD1), and the collection $X$ in \eqref{eq:21} represents a realization of $(P_t)$. In this section, we make an additional assumption that $X$ satisfies (HD2) as defined in Definition~\ref{DEFA7}. This assumption ensures that the augmented system
\begin{equation}\label{eq:30}
X=(\Omega, \cF,\cF_t,X_t,\theta_t, \bP^x)
\end{equation}
is a right process on $E$, and $(P_t)$ is a right semigroup, as defined in Definition~\ref{DEFA9}. In other words, $X$ remains the canonical realization of $(P_t)$, but with the filtration replaced by the augmented one. Moreover, for every $\alpha > 0$ and $f \in \cS^\alpha$, the function $t \mapsto f(X_t)$ is right continuous $\bP^x$-almost surely for every $x \in E$. Here, $\cS^\alpha$ denotes the family of all $\alpha$-excessive functions, where $f \in p\cE^u$ is in $\cS^\alpha$ if and only if $e^{-\alpha t}P_t f \uparrow f$ as $t \downarrow 0$. This additional assumption (HD2) implies the strong Markov property of $X$ and the right continuity of the augmented natural filtration $(\cF_t)$. The definitions of quasi-absorbing sets and quasi-polar sets remain the same as in Definition~\ref{DEF21}, thanks to Lemma~\ref{LM23}.

In the appendix of \cite{BCR20}, some interesting observations are made regarding the (Borel) right process with respect to the \emph{natural topology}, which encompasses topologies that are coarser than the fine topology, including the original topology on $E$. It is shown that the (Borel) right process under the natural topology always has almost surely right-continuous trajectories. This observation is significant as it allows us to relax the restrictions on the topology of $E$ when considering various problems related to (Borel) right process. Similarly, in the discussion of this section, we will adopt this approach and consider the right process without placing strong emphasis on the original topology of $E$. It is important to note that the original topology on $E$ is solely used to define right-continuous trajectories in (HD1) and does not play a role in the measurable structures or (HD2).

\subsection{Weakly $U$-negligible set}

{
Let $N\in \cE^u$ be a weakly $U$-negligible set and put $\tilde{E}:=E\setminus N$.  Denote by $(\tilde{U}^\alpha)_{\alpha>0}$ the restriction of $(U^\alpha)_{\alpha>0}$ to $\tilde{E}:=E\setminus N$.  
 Instead of equipping $\tilde{E}$ with the subspace topology,  we consider 
natural topologies relative to $(\tilde{U}^\alpha)$ in the following sense.

\begin{definition}
\begin{itemize}
\item[(1)] Given $\alpha>0$,  $\tilde{f}\in p\tilde{\cE}^u$ is called \emph{$\tilde{U}^\alpha$-excessive} provided that $\beta \tilde{U}^{\alpha+\beta}\tilde{f}\leq \tilde{f}$ for all $\beta>0$ and $\lim_{\beta\rightarrow \infty}\beta\tilde{U}^{\alpha+\beta}\tilde{f}=\tilde{f}$.  Denote by $\tilde{\cS}^\alpha$ the family of all $\tilde{U}^\alpha$-excessive function.  
\item[(2)] The \emph{$\tilde{U}$-fine topology} $\tilde{\tau}_f$ on $\tilde E$ is defined as the coarsest topology on $\tilde{E}$ making all functions in $\cup_{\alpha>0}\tilde{\cS}^\alpha$ continuous.  
\item[(3)] A topology $\tilde{\tau}$ on $\tilde{E}$ is called \emph{natural},   if it is Radonian,  i.e.  $(\tilde{E},\tilde{\tau})$ is homoemorphic to a universally measurable set (equipped with the subspace topology) of some compact metric space,  generates the Borel $\sigma$-algebra $\mathcal{B}(\tilde{\tau})=\tilde{\cE}:=\cE|_{\tilde{E}}$ and $\tilde{\tau}\subset \tilde{\tau}_f$. 
\item[(4)] A Markov process (resp.  transition function) $\tilde X$ (resp.  $(\tilde{P}_t)$) on $\tilde{E}$ is called a right process (resp.  right semigroup) relative to some topology $\tilde{\tau}$ on $\tilde{E}$ provided that $\tilde{\tau}$ generates the Borel $\sigma$-algebra $\mathcal{B}(\tilde{\tau})=\cE|_{\tilde{E}}$ and $\tilde{X}$ (resp. $(\tilde{P}_t)$) satisfies (HD1)  and (HD2) relative to $\tilde{\tau}$. 
\end{itemize}
\end{definition}
\begin{remark}
\begin{itemize}
\item[(1)] When $\tilde{U}^\alpha$ is the resolvent of some right process on $\tilde{E}$,  the definitions of $\tilde{U}^\alpha$-excessive functions and $\tilde{U}$-fine topology coincide with those of usual excessive functions and fine topology.  
\item[(2)] The universally measurable $\sigma$-algebra on $\tilde{E}$ relative to arbitrary natural topology $\tilde{\tau}$ is  $\mathcal{B}^u(\tilde{\tau})=\tilde{\cE}^u:=\cE^u|_{\tilde{E}}$. 
\item[(3)] The subspace topology $\tilde{\tau}_0$ on $\tilde{E}$ (of the original topology on $E$) is natural.  To see this,  let $\tau$ and $\tau_f$ be the original topology and fine topology on $E$.   Since the restriction of $\alpha$-excessive function for $X$ to $\tilde{E}$ is $\tilde{U}^\alpha$-excessive,  it follows that $\tau_f|_{\tilde{E}}\subset \tilde{\tau}_f$,  where $\tau_f|_{\tilde{E}}$ stands for the subspace topology of $\tau_f$ on $\tilde{E}$.  Therefore $\tilde{\tau}_0=\tau|_{\tilde{E}}\subset \tau_f|_{\tilde{E}}\subset \tilde{\tau}_f$,  as arrives at the conclusion.   
\item[(4)] If $(\tilde{P}_t)$ is a right semigroup relative to some topology $\tilde{\tau}$  on $\tilde{E}$ whose resolvent is $(\tilde{U}^\alpha)$,  then $\tilde{\tau}$ must be natural.  
\end{itemize}
\end{remark}
 


Having established the necessary setups, we are now poised to unveil one of the principal results in this section.

\begin{theorem}\label{THM31}
Let $N\in \cE^u$ be weakly $U$-negligible and $(\tilde{U}^\alpha)_{\alpha>0}$ be the restriction of $(U^\alpha)_{\alpha>0}$ to $\tilde{E}:=E\setminus N$.  The following are equivalent:
\begin{itemize}
\item[(i)] There is a right semigroup $(\tilde{P}_t)_{t\geq 0}$ on $\tilde{E}$ relative to the subspace topology,  whose resolvent is $(\tilde{U}^\alpha)_{\alpha>0}$.
\item[(ii)] There is a right semigroup $(\tilde{P}_t)_{t\geq 0}$ on $\tilde{E}$ relative to some natural topology,  whose resolvent is $(\tilde{U}^\alpha)_{\alpha>0}$.
\item[(iii)] There is a right semigroup $(\tilde{P}_t)_{t\geq 0}$ on $\tilde{E}$ relative to all natural topology,  whose resolvent is $(\tilde{U}^\alpha)_{\alpha>0}$.
\item[(iv)] $\tilde E$ is quasi-absorbing for $X$.  
\end{itemize}
Meanwhile all the right semigroups above are identical to the restriction of $(P_t)$ to $\tilde{E}$,  i.e.  $\tilde{P}_t(x,B)=P_t(x,B)$ for $t\geq 0,  x\in \tilde{E}$ and $B\in \tilde{\cE}^u$.  
\end{theorem}
\begin{proof}
Obviously (iii) implies (i),  and (i) implies (ii).  The assertion that (iv) implies (i) has already been proved in \cite[Theorem~(12.30)]{Sh88},  The contrary that (i) implies (iv) is a consequence of Theorem~\ref{THM24}.  It suffices to argue that (ii) implies (iii).  

Suppose that (ii) holds true and that $\tilde{\tau}$ is such a natural topology on $\tilde{E}$.   Let
\[
	\tilde X=(\tilde \Omega, \tilde \cF,\tilde \cF_t,\tilde X_t,\tilde \theta_t, \tilde \bP^x)
\]
be the augmented canonical realization of $\tilde{P}_t$ relative to $\tilde{\tau}$.  The family of null sets is denoted by $\tilde{\cN}$ (see Appendix~\ref{APPA5}).  The fine topology of $\tilde{X}$ is actually the $\tilde{U}$-fine topology $\tilde{\tau}_f$.  

Take another natural topology $\tilde{\tau}_1$ on $\tilde{E}$.  Let $(\hat{E}_1, \tilde{d}_1)$ be a compact metric space in which $(\tilde{E},\tilde{\tau}_1)$ is embedded.  Denote by $C_{\tilde{d}_1}(\tilde{E})$ the family of all $\tilde{d}_1$-uniformly continuous functions on $\tilde{E}$.  Take a family $\tilde{\cC}_1\subset C_{\tilde{d}_1}(\tilde{E})$ of countably many functions such that $\tilde{\cC}_1$ is dense in $C_{\tilde{d}_1}(\tilde{E})$.  (Such a family exists because $C_{\tilde{d}_1}(\tilde{E})=C(\hat{E}_1)|_{\tilde{E}}$ and $C(\hat{E}_1)$ is separable.) In view of page 251 in \cite{RF10},  $\tilde{\cC}_1$ generates the topology $\tilde{\tau}_1$.  In other words,  $\tilde{E} \ni x_n\rightarrow x\in \tilde{E}$ relative to $\tilde{\tau}_1$ if and only if $f(x_n)\rightarrow f(x)$ for all $f\in \tilde \cC_1$. 

Since $\tilde{\tau}_1\subset \tilde{\tau}_f$,  every $f\in \tilde{\cC}_1$ is $\tilde{\tau}_f$-continuous.  By means of,  e.g.,  \cite[(10.18)]{Sh88},  $t\mapsto f(\tilde X_t)$ is a.s.  right continuous for $f\in \tilde{\cC}_1$.  In other words,  there is a null set $\tilde\Gamma_f\in \tilde \cN$ such that $t\mapsto f(\tilde X_t(\omega))$ is right continuous for every $\omega\in \tilde \Omega\setminus \tilde \Gamma_f$.  
Set $\tilde\Gamma:=\cup_{f\in \cC}\tilde \Gamma_f$.   Since $\tilde\cC$ is countable,  it follows that $\tilde \Gamma\in\tilde \cN$.  For any $\omega\in \tilde{\Omega}\setminus \tilde{\Gamma}$,  since $t\mapsto f(\tilde{X}_t(\omega))$ is right continuous for all $f\in \tilde{\cC}_1$,  one obtains that $t\mapsto \tilde{X}_t(\omega)$ is right continuous relative to  $\tilde{\tau}_1$ for every $\omega\in \tilde{\Omega}\setminus \tilde{\Gamma}$.  Define $\tilde{\Omega}_1:=\tilde{\Omega}\setminus \tilde{\Gamma}$,  and let $(\tilde{\cF}^1,\tilde{\cF}^1_t,\tilde{X}^1_t, \tilde{\bP}^x_1)$ be the restriction of $(\tilde{\Omega},\tilde{\cF},\tilde{\cF}_t, \tilde{X}_t, \tilde{\bP}^x)$ to $\tilde{\Omega}_1$.   Obviously $\tilde{\theta}_t{\tilde{\Omega}_1}\subset \tilde{\Omega}_1$.  Denote its restriction to $\tilde{\Omega}_1$ by $\tilde{\theta}^1_t$.   It is straightforward to verify that
\[
	\tilde{X}^1=(\tilde \Omega^1, \tilde \cF^1,\tilde \cF_t^1,\tilde X_t^1,\tilde \theta_t^1, \tilde \bP^x_1)
\]
is a realization of $\tilde{P}_t$ relative to $\tilde{\tau}_1$.  Therefore $\tilde{P}_t$ also satisfies (HD1) relative to $\tilde{\tau}_1$.  Specifically $(\tilde{P}_t)$ is a right semigroup relative to $\tilde{\tau}_1$.  

The identification between $\tilde{P}_t$ and the restriction of $P_t$ to $\tilde{E}$ can be easily concluded by virtue of Theorem~\ref{THM24}.  That completes the proof. 
\end{proof}
}

\subsection{$U$-negligible set}

Let us now investigate the restriction $\tilde{U}^\alpha$ of $U^\alpha$ to the complement $\tilde{E}=E\setminus N$ of a $U$-negligible set $N\in \cE^u$. This stronger assumption on $N$ has an important implication: $N$ does not contain any non-empty finely open subsets, as shown in \cite[(10.12)]{Sh88}. (In contrast, a weakly $U$-negligible set may contain non-empty finely open subsets. For example, in the example mentioned in Remark~\ref{RM22}~(2), the set $\{2\}$ is both weakly $U$-negligible and finely open.)  As a result,  we have less chances to build a right process for $\tilde{U}^\alpha$.  

{
\begin{theorem}\label{THM32}
Let $N\in \cE^u$ be $U$-negligible and $(\tilde{U}^\alpha)_{\alpha>0}$ be the restriction of $(U^\alpha)_{\alpha>0}$ to $\tilde{E}:=E\setminus N$.  The following are equivalent:
\begin{itemize}
\item[(i)] There is a right semigroup $(\tilde{P}_t)_{t\geq 0}$ on $\tilde{E}$ relative to the subspace topology,  whose resolvent is $(\tilde{U}^\alpha)_{\alpha>0}$.
\item[(ii)] There is a right semigroup $(\tilde{P}_t)_{t\geq 0}$ on $\tilde{E}$ relative to some natural topology,  whose resolvent is $(\tilde{U}^\alpha)_{\alpha>0}$.
\item[(iii)] There is a right semigroup $(\tilde{P}_t)_{t\geq 0}$ on $\tilde{E}$ relative to all natural topology,  whose resolvent is $(\tilde{U}^\alpha)_{\alpha>0}$.
\item[(iv)] $N$ is quasi-polar for $X$.  
\end{itemize}
Meanwhile all the right semigroups above are identical to the restriction of $(P_t)$ to $\tilde{E}$,  i.e.  $\tilde{P}_t(x,B)=P_t(x,B)$ for $t\geq 0,  x\in \tilde{E}$ and $B\in \tilde{\cE}^u$.  
\end{theorem}}
\begin{proof}
{The equivalences between (i),  (ii) and (iii) have been obtained in Theorem~\ref{THM31}.  
Since the complement of a quasi-polar set is quasi-absorbing,  (iv) is stronger than the other three due to Theorem~\ref{THM31}.  We only need to show that (i) implies (iv). } Let
\[
	\tilde X=(\tilde \Omega, \tilde \cF,\tilde \cF_t,\tilde X_t,\tilde \theta_t, \tilde \bP^x)
\]
be the augmented canonical realization of $\tilde{P}_t$ relative to the subspace topology.  We will adopt a well-known \emph{Ray-Knight method} to complete the proof;  see \cite{Ge75} and \cite[Chapter V]{Sh88} for more details about this method.  A brief summary is also reviewed in Appendix~\ref{APPB} for readers' convenience.  

Take a pre-Ray class $\cC\subset pC_d(E)$ for the right process $X$ (see Definition~\ref{DEFB1}).  The rational Ray cone generated by $(U^\alpha)$ and $\cC$ is denoted by $\cR$.  It is straightforward to verify that $\tilde\cC:=\{f|_{\tilde{E}}: f\in \cC\}$ satisfies all conditions in Definition~\ref{DEFB1} with $\tilde{E}$ in place of $E$.  Hence $\tilde{\cC}$ is a pre-Ray class for the right process $\tilde{X}$.  Denote by $\tilde{\cR}$ the rational Ray cone generated by $(\tilde{U}^\alpha)$ and $\tilde{\cC}$.  Checking the construction procedures of rational Ray cone after Definition~\ref{DEFB1} and noting that $N$ is $U$-negligible,  one may easily obtain $\tilde{\cR}=\{f|_{\tilde{E}}:f\in\cR\}$.  Certainly both $\cR$ and $\tilde{\cR}$ are countable.  Hence we may write 
\[
	\cR=\{g_n:n\geq 1\},\quad \tilde{\cR}=\{\tilde{g}_n: n\geq 1\},
\]
where $\tilde{g}_n=g_n|_{\tilde{E}}$. 
 We argue that 
 \begin{equation}\label{eq:32-2}
 \|g_n\|:=\sup_{x\in E}|g_n(x)|=\sup_{x\in \tilde{E}}|\tilde{g}_n(x)|=:\|\tilde{g}_n\|.
 \end{equation}  
 Let $\tau_f$ be the fine topology on $E$ relative to $X$.  Note that $g_n\in \cR\subset \cup_{\alpha>0}b\cS^\alpha$.  Hence $g_n$ is $\tau_f$-continuous.  The assumption that $N$ is $U$-negligible implies that $N$ contains no non-empty finely open subsets.  In particular,  for any $x\in N$,  there is a sequence $x_k\in \tilde{E}$ such that $x_k\rightarrow x$ as $k\rightarrow \infty$ in $\tau_f$.  Since $g_n(x_k)\rightarrow g_n(x)$ as $k\rightarrow\infty$,  we can obtain that $\|g_n\|=\sup_{x\in \tilde{E}}|g_n(x)|=\|\tilde{g}_n\|$,  as arrives at \eqref{eq:32-2}.  Then examining the construction of Ray-Knight compactification in Appendix~\ref{APPB} allows us to conclude that $(E, d, U^\alpha)$ with rational Ray cone $\cR$ and $(\tilde{E},d, \tilde{U}^\alpha)$ with rational Ray cone $\tilde{\cR}$ have the same Ray-Knight compactification $(\bar{E},\bar{\rho},\bar{U}^\alpha)$.  
 
 Let $D$ be the set of non-branching points of $\bar{U}^\alpha$ and 
 \[
 	\tilde E_R:=\{x\in \bar{E}: \bar{U}^\alpha(x,\cdot)\text{ is carried by }\tilde{E}\}
 \]
 be the Ray space relative to $\tilde{X}$.  Applying \cite[(17.14)]{Sh88} to $\bar{U}^\alpha$ and $X$,  one gets that for $x\in E$,  $\bar{U}^\alpha(x,\cdot)$ is carried by $E$ and the restriction of $\bar{U}^\alpha(x,\cdot)$  to $E$ is equal to $U^\alpha(x,\cdot)$.  Since $N$ is $U$-negligible,  $U^\alpha(x,\cdot)$,  hence $\bar{U}^\alpha(x,\cdot)$,  is carried by $\tilde{E}$ for $x\in E$.  This yields $\tilde{E}\subset E\subset \tilde{E}_R$.  
 Put $\tilde{E}_D:=D\cap \tilde{E}_R$.  Since $E\subset D$,  it follows that 
 \begin{equation}\label{eq:31}
 \tilde{E}\subset E\subset \tilde{E}_D.  
 \end{equation}
 By virtue of \cite[Theorem~39.15]{Sh88} for $\bar{U}^\alpha$ and $\tilde{X}$,  the restriction of $\bar{U}^\alpha$ (or its Ray semigroup $\bar{P}_t$) to $\tilde{E}_D$ corresponds to a right semigroup,  denoted by $\tilde{Q}_t$,  on $\tilde{E}_D$ and in addition,  $\tilde{E}_D\setminus \tilde{E}$ is quasi-polar for the right process on $\tilde{E}_D$ with transition semigroup $\tilde{Q}_t$.  Let 
 \begin{equation}\label{eq:32}
 \tilde{Y}=(\tilde{W}, \tilde{\cW},\tilde{\cW}_t, \tilde{Y}_t,  \tilde \vartheta_t,  \tilde{\mathbf{Q}}^x)
 \end{equation}
 be the augmented  canonical realization of $\tilde{Q}_t$.  On account of \eqref{eq:31},  both $\tilde{E}_D\setminus E$ and $E\setminus \tilde{E}$ are quasi-polar sets for $\tilde{Y}$.  Particularly,  $E$ is quasi-absorbing for $\tilde{Y}$.  Restricting \eqref{eq:32} to $\tilde{W}^1:=\{\omega\in \tilde{W}: \tilde{Y}_t(\omega)\in E,\forall t\geq 0\}$ and applying \cite[Theorem~(12.30)]{Sh88},  one gets a new right process 
  \begin{equation}\label{eq:35}
 \tilde{Y}^1=(\tilde{W}^1, \tilde{\cW}^1,\tilde{\cW}^1_t, \tilde{Y}^1_t,  \tilde \vartheta^1_t,  \tilde{\mathbf{Q}}_1^x)
 \end{equation}
 on $E$ whose semigroup is the restriction of $\tilde{Q}_t$ to $E$.  The restriction of $\tilde{Q}_t$ to $E$ is actually the restriction of Ray semigroup $\bar{P}_t$ to $E$.  According to \cite[Theorem~(17.16)]{Sh88},  it is exactly equal to $P_t$.  In other words,  $\tilde{Y}^1$ is indeed another realization of $X$.  Take a probability measure $\mu$ on $E$.  
Define
\[
\begin{aligned}
	\Gamma_1&:=\{\omega\in \tilde{W}^1: \tilde{Y}^1_t(\omega)\notin N,\forall t> 0\}  \\
	&=\{\omega\in \tilde{W}: \tilde{Y}_t(\omega)\in  \tilde{E},\forall t> 0, \tilde{Y}_0(\omega)\in E\}=:\Gamma.  
\end{aligned}\]
Denote the outer measures of $(\tilde{W},\tilde{\cW}, \tilde{\mathbf{Q}}^\mu)$ and $(\tilde{W}^1,\tilde{\cW}^1, \tilde{\mathbf{Q}}^\mu_1)$ by $\tilde{\mathbf{Q}}^\mu_*$ and $\tilde{\mathbf{Q}}^\mu_{1*}$.  Since $\tilde{\mathbf Q}^\mu\{\tilde{Y}_0\in E\}=1$ and $\tilde{\mathbf{Q}}^\mu_*\{\tilde{Y}_t(\omega)\in  \tilde{E},\forall t> 0\}=1$ due to that $\tilde{E}_D\setminus \tilde{E}$ is quasi-polar for $\tilde{Y}$,  one has $\tilde{\mathbf{Q}}^\mu_*(\Gamma)=1$.  We argue $\tilde{\mathbf{Q}}^\mu_{1*}(\Gamma_1)=1$.  In fact,  take arbitrary $B^1\in \tilde{\cW}^1$ with $\Gamma_1\subset B^1$.  There is some $B\in \tilde{\cW}$ such that $B^1=B\cap \tilde{W}^1$.  Then $\tilde{\mathbf{Q}}^\mu_*(\Gamma)=1$ and $\Gamma=\Gamma_1\subset B^1\subset B$ imply $\tilde{\mathbf Q}^\mu_1(B^1)=\tilde{\mathbf Q}^\mu(B)=1$.  Hence $\tilde{\mathbf{Q}}^\mu_{1*}(\Gamma_1)=1$ holds true.  

Recall that \eqref{eq:30} is the augmented canonical realization of $(P_t)$,  and we already show that \eqref{eq:35} is another realization of $(P_t)$.  Define a map $\Phi: \tilde{W}^1\rightarrow \Omega$,  as is characterized by $\tilde Y^1_t=X_t\circ \Phi$; see Remark~\ref{RMA4}~(3).  Then $\Phi\in \tilde{\cW}^1/\cF^u$ and $\tilde{\mathbf Q}^\mu_1\circ \Phi^{-1}=\bP^\mu$.  Put
\[
	\Gamma_0:=\{\omega\in \Omega: X_t(\omega)\notin N,\forall t>0\}.  
\]
For any $B\in \cF^u$ with $\Gamma_0\subset B$,  we have $\Gamma_1\subset \Phi^{-1}(B)\in \tilde{\cW}^1$.  Thus $\tilde{\mathbf{Q}}^\mu_{1*}(\Gamma_1)=1$ yields $\bP^\mu(B)=\tilde{\mathbf Q}^\mu_1\circ \Phi^{-1}(B)=1$.  As a result,  $\bP^\mu_*(\Gamma_0)=1$,  as arrives at  the conclusion that $N$ is quasi-polar for $X$.
\end{proof}

This result,  together with Theorem~\ref{THM31},  readily implies the following.

\begin{corollary}
Let $N\in \cE^u$ be $U$-negligible.  Then $N$ is quasi-polar,  if and only if $E\setminus N$ is quasi-absorbing.  
\end{corollary}

We close this section with two examples in terms of non-Borel universally measurable subsets of $\bR^n$.  

\begin{example}\label{EXA34}
Let $X$ be the Brownian motion on $\bR^n$ with $n\geq 2$ and $U^\alpha$ be its resolvent.  Salisbury \cite{Sa87} constructed a universally measurable set $Z\subset \bR^n$ such that $Z$ is of zero Lebesgue measure,  every non-constant continuous path hits $Z$ and every Borel subset of $Z$ is polar. Obviously $Z$ is $U$-negligible and $Z$ is not quasi-polar.  ($\bR^n\setminus Z$ is not quasi-absorbing either.)  Hence there is no right process on $\bR^n\setminus Z$ whose resolvent is the restriction of $U^\alpha$ to $\bR^n\setminus Z$.  (There also exists no transition function satisfying (HD1) whose resolvent is this restriction if $\bR^n\setminus Z$ is endowed with the subspace topology.)

In addition,  Salisbury \cite{Sa87} also constructed another universally measurable set $Z'\subset \bR^n$ for $n\geq 4$ such that it is of zero Lebesgue measure and the $\bP^\mu$-outer measures of both $\{X_t\in Z': \exists t>0\}$ and $\{X_t\notin Z': \forall t>0\}$ are equal to $1$ for any probability measure $\mu$ on $\bR^n$.  In particular,  $Z'$ is quasi-polar (hence any Borel set contained in $Z'$ is polar),  while any Borel set containing $Z'$ is not polar (in fact,  if $B\in \mathcal{B}(\bR^n)$ with $Z'\subset B$,  then $\bP^\mu(T_B<\infty)=1$ where $T_B:=\inf\{t>0: X_t\in B\}$).  As a result,  there is a right process on $\bR^n\setminus Z'$ whose resolvent is the restriction of $U^\alpha$ to $\bR^n\setminus Z'$.  
\end{example}

\section{Restriction of semi-Dirichlet form associated to a right process}

Let $X=(\Omega, \cF,\cF_t,X_t,\theta_t,\mathbf{P}^x)$ be a right process on a Radon space $E$ whose right semigroup is $(P_t)$ and whose resolvent is $(U^\alpha)$.  Write $U$ for $U^\alpha$ with $\alpha=0$.  We will regard the ceremony $\Delta$ as a point outside $E$ and hence the lifetime $\zeta=\inf\{t>0:X_t=\Delta\}$ should be also attaching to $X$.  In this case $(P_t)$ is,  in general,  sub-Markov on $(E,\cE^u)$,  every function on $E$ is automatically extended to $E_\Delta:=E\cup \{\Delta\}$ by setting $f(\Delta)=0$,  and $X$ is called a right process on $E$ with lifetime $\zeta$ and transition semigroup $(P_t)$; see Appendix~\ref{APPA8}.  Furthermore, the definition of a quasi-absorbing set needs to be adjusted by considering the time before $\zeta$. Specifically, a set $F\in \mathcal{E}^u$ is called quasi-absorbing if, for any initial law $\mu$ supported on $F$, the set $\{X_t\in F, \forall 0\leq t<\zeta\}$ has full $\mathbf{P}^\mu$-outer measure. Equivalently, $F\cup \{\Delta\}$ is quasi-absorbing for the Markov extension of $(P_t)$ to $E_\Delta$. (The definition of quasi-polar sets remains unchanged.)

In this section, we will introduce two additional conditions on $X$, which are adapted from Fitzsimmons \cite{Fi01}.

\begin{hypothesis}
The following hypotheses are assumed to hold true for $X$:
\begin{itemize}
\item[(H1)] (Sector condition.) There is a $\sigma$-finite positive measure $\fm$ on $(E,\cE)$ such that for any $t\geq 0$,  the restriction of $P_t$ to $L^2(E,\fm)\cap b\cE$ extends uniquely to a contraction operating in $L^2(E,\fm)$ and its infinitesimal generator $Lf:=\lim_{t\rightarrow 0}(P_tf-f)/t$ with domain $D(L)$ (i.e.  the class of functions $f\in L^2(E,\fm)$ for which the indicated limit exists in the strong sense in $L^2(E,\fm)$) satisfies the sector condition: There is a finite constant $K\geq 1$ such that the bilinear form $\sE(f,g):=(f,-Lg)$ with $f,g\in D(L)$ satisfies the estimate
\[
	|\sE(f,g)|\leq K\cdot (f,f-Lf)^{1/2} (g,g-Lg)^{-1/2},\quad \forall f,g\in D(L),
\] 
where $(\cdot,\cdot)$ stands for the inner product of $L^2(E,\fm)$.  
\item[(H2)] ($\fm$-tightness.) There is an increasing sequence $(K_n)$ of compact subsets of $E$ such that 
\[
\lim_{n\rightarrow \infty}\bP^\fm\{T_{E\setminus K_n}<\zeta\}=0,
\]
where the notation $\bP^m$ means that for $\fm$-a.e.  $x\in E$,  the $\bP^x$-measure of this set is $0$,  and $T_{E\setminus K_n}:=\inf\{t>0:X_t\in E\setminus K_n\}$.  
\end{itemize}
\end{hypothesis}
\begin{remark}
Fitzsimmons \cite{Fi01} also included the assumption of transience, which states that there exists a strictly positive function $g\in b\cE^u$ such that $Ug$ is finite everywhere. However, this assumption can be dropped because if $X$ is not transient, we can make it transient by introducing an exponential killing rate, thereby ensuring that all results established in \cite{Fi01} apply to the killed process, and the corresponding results for the original process can be easily derived as well; see \cite[Remark~2.8~(b)]{Fi01}. The hypothesis (H2) is replaced by a stronger one in \cite{Fi01}: $E$ is a metrizable \emph{co-Souslin space}.  For convenience we take the current version mentioned in \cite[(2.1)']{Fi01}. 
\end{remark}

Assuming hypothesis (H1), let $\sF$ be the completion of $D(L)$ with respect to the norm $\mathbf{e}(f) := (f,f-Lf)^{1/2}$. Then $(\sE,\sF)$ constitutes a semi-Dirichlet form on $L^2(E,\fm)$ as defined in \cite{MOR95}. The main result of \cite{Fi01} further demonstrates that under the conditions (H1) and (H2), the right process $X$ is necessarily $\fm$-special and $\fm$-standard. Additionally, it implies that $(\sE,\sF)$ is quasi-regular according to \cite{MOR95}.

\begin{theorem}[Fitzsimmons \cite{Fi01}]\label{THM43}
Let $X$ be a right process on a Radon space $E$ such that (H1) and (H2) hold.  Then its semi-Dirichlet form $(\sE,\sF)$ on $L^2(E,\fm)$ is quasi-regular,  and $X$ is properly associated with $(\sE,\sF)$ in the sense that $P_t f$ is quasi-continuous for each $f\in bp\cE\cap L^2(E,\fm)$.  
\end{theorem}

We will refrain from restating the definitions of (semi-)Dirichlet form,  quasi-regularity or quasi-notions relative to a (semi-)Dirichlet form.  The readers can refer to the literatures such as \cite{Fi01,  MOR95,  MR92}.  Quasi-regularity holds significant importance as it serves as both a sufficient and necessary condition for a (semi-)Dirichlet form to be associated with a certain \emph{standard} Markov process.

The objective of this section is to investigate whether $(\sE,\sF)$ remains quasi-regular when it is defined on $L^2(E\setminus N,\fm)$,  where $N\subset E$ is an $\fm$-negligible set,  i.e.  $\fm(N)=0$.  It is important to note that $(\sE,\sF)$ still constitutes a semi-Dirichlet form on $L^2(E\setminus N,\fm)$, as the semi-Dirichlet form is defined based on $\fm$-equivalence classes. Furthermore, it is worth mentioning that $N$ belongs to $\bar{\cE}^\fm\supset \cE^u$, where $\bar{\cE}^\fm$ represents the completion of $\cE$ with respect to $\fm$,  and $N$ is not necessarily universally measurable in general.

Let us first examine the restriction of $U^\alpha$ outside a weakly $U$-negligible set $N\in \cE^u$ with $\fm(N)=0$.  Thanks to Theorem~\ref{THM31},  this restriction corresponds to a right process $\tilde{X}$ whenever $E\setminus N$ is quasi-absorbing.  Meanwhile the restricted right process still satisfies (H1).  However, in order to ensure the quasi-regularity of its semi-Dirichlet form, we also require $\fm$-tightness for $\tilde{X}$, as indicated by Theorem~\ref{THM43}.  It is important to note that $\fm$-tightness on a process implies that the process remains within a Lusin space consisting of an increasing sequence of compact subsets. Consequently, $\tilde{X}$ becomes a Borel right process on a smaller absorbing set, which is the complement of a larger Borel measurable negligible set. This particular case is the only scenario where quasi-regularity of the semi-Dirichlet form can be obtained through restriction, as illustrated in Theorem~\ref{THM44}.



A set $B\in \cE^e$, the $\sigma$-algebra generated by all ($\alpha$-)excessive functions,  is called \emph{$\fm$-inessential},  if $\fm(B)=0$ and $E\setminus B$ is (quasi-)absorbing for $X$ (see,  e.g.,  \cite[Definition~3.16~(a)]{Fi01}).  Now we have a position to present the main result of this section.  

\begin{theorem}\label{THM44}
Let $X$ be a right process on a Radon space $E$ such that (H1) and (H2) hold,  whose quasi-regular semi-Dirichlet form on $L^2(E,\fm)$ is $(\sE,\sF)$.  Let $N\subset E$ be an $\fm$-negligible set and $\tilde{E}:=E\setminus N$ endowed with the subspace topology of $E$.  Then $(\sE,\sF)$ is quasi-regular on $L^2(\tilde{E},\fm)$,  if and only if $N\subset B$ for some $\fm$-inessential set $B\in \cE$.    
\end{theorem}
\begin{proof}
The sufficiency is trivial because in this case $B$ is \emph{$\fm$-polar} and hence $N$ is $\sE$-exceptional.  

We argue the necessity and the idea of the proof is through the line of \cite[\S2.3]{BCR22}.  Suppose that $(\sE,\sF)$ is quasi-regular on $L^2(\tilde{E},\fm)$.  Write $(\tilde{\sE},\tilde{\sF})$ for $(\sE,\sF)$ if it is put on  $L^2(\tilde{E},\fm)$.  Denote by $\tilde{\cE}=\cE|_{\tilde{E}}$ the Borel $\sigma$-algebra on $\tilde{E}$.  (Attention that $\tilde{B}\in \tilde{\cE}$ is not necessarily in $\cE$.) On account of \cite[Theorem~3.8]{MOR95},  there is an $\tilde{\sE}$-exceptional set $\tilde{B}\in \tilde{\cE}$ and a Borel right process $\tilde{X}$ on the Lusin topological space $\tilde{E}_1:=\tilde{E}\setminus \tilde{B}$ properly associated with $(\tilde{\sE},\tilde{\sF})$.  The Borel $\sigma$-algebra on $\tilde{E}_1$ is $\tilde{\cE}_1=\tilde{\cE}|_{\tilde{E}_1}=\cE|_{\tilde{E}_1}$.  Denote by $\tilde{P}_t$ and $\tilde{U}^\alpha$ the transition semigroup and resolvent of $\tilde{X}$ on $(\tilde{E}_1,\tilde{\cE}_1)$.   The embedding map $\tilde{i}: \tilde{E}_1\rightarrow E$ is obviously $\tilde{\cE}_1/\cE$ measurable.  Since $\tilde{E}_1$ is Lusin,  it follows from Lusin's theorem (e.g.,  \cite[A2.6]{Sh88}) that $\tilde{E}_1\in \cE$.  Particularly  $\tilde{\cE}_1\subset \cE$.  

Put $(\tilde{\sE},\tilde{\sF})$ on $L^2(\tilde{E}_1,\fm)$ and still denoted it by $(\tilde{\sE},\tilde{\sF})$.  Obviously $(\tilde{\sE},\tilde{\sF})$ remains quasi-regular on $L^2(\tilde{E}_1,\fm)$.  Note that if $\{F_n\subset E:n\geq 1\}$ is an $\sE$-nest,  then $\{\tilde{F}_n:=F_n\cap \tilde{E}_1: n\geq 1\}$ is an $\tilde{\sE}$-nest.  Hence $\tilde{f}:=f|_{\tilde{E}_1}$ is $\tilde{\sE}$-quasi-continuous for every $\sE$-quasi-continuous function $f$ on $E$.  Take a countable dense subset $\cC$ of $C_d(E)$.  \cite[Proposition~3.4]{MOR95} yields $U^\alpha f$ is $\sE$-quasi-continuous for $f\in \cC$, and hence $U^\alpha f|_{\tilde{E}_1}$ is $\tilde{\sE}$-quasi-continuous.  Since $U^\alpha f=\tilde{U}^\alpha \tilde{f}$,  $\fm$-a.e.,  where $\tilde{f}:=f|_{\tilde{E}_1}$,  and $\tilde{U}^\alpha \tilde{f}$ is $\tilde{\sE}$-quasi-continuous,  it follows that $U^\alpha f|_{\tilde{E}_1}=\tilde{U}^\alpha \tilde{f}$ outside an $\tilde{\sE}$-exceptional set $\tilde{B}_{\alpha, f}\in \tilde{\cE}_1$.  Let $\tilde{B}_1\in \tilde{\cE}_1$ be an $\fm$-inessential set of $\tilde{X}$ containing the union of the sets $\tilde{B}_{\alpha,f}$ as $\alpha$ varies over all strictly positive rationals and as $f$ varies over $\cC$.  The existence of such $\tilde{B}_1$ is due to,  e.g.,  \cite[Theorem~A.2.15]{CF12}. Set $\tilde{E}_2:=\tilde{E}_1\setminus \tilde{B}_1$ with its Borel $\sigma$-algebra $\tilde{\cE}_2=\cE|_{\tilde{E}_2}$.  Then it is easy to obtain 
\begin{equation}\label{eq:41}
	U^\alpha f|_{\tilde{E}_2}=\tilde{U}^\alpha \tilde{f}|_{\tilde{E}_2},\quad \forall \alpha>0,  f\in C_d(E).  
\end{equation}
Since $\tilde{E}_2$ is absorbing for $\tilde{X}$,  it follows that $\tilde{B}_1$ is weakly $\tilde{U}$-negligible and the restriction $\tilde{V}^\alpha$ of $\tilde{U}^\alpha$ to $\tilde{E}_2$ is the resolvent of a Borel right process on $\tilde{E}_2$.  
A monotone class argument based on \eqref{eq:41} allows us to conclude that $E\setminus \tilde{E}_2$ is weakly $U$-negligible and the restriction of $U^\alpha$ to $\tilde{E}_2\in \cE$ is equal to $\tilde{V}^\alpha$.  Particularly,  the restriction of $U^\alpha$ to $\tilde{E}_2$ corresponds to a right process.  On account of Theorem~\ref{THM31} and  $\tilde{E}_2\in \cE$,  one gets that $\tilde{E}_2$ is absorbing for $X$.  Note that $\fm(E\setminus \tilde{E}_2)=0$.  We eventually have an $\fm$-inessential set $B:=E\setminus \tilde{E}_2\in \cE$ with $N\subset B$.  That completes the proof. 
\end{proof}

This characterization can be applied to build examples of non-quasi-regular (semi-)Dirichlet forms,  as Beznea et al. did in \cite{BCR22}.   In summary,  given a quasi-regular (semi-)Dirichlet form on $L^2(E,\fm)$ and a non-$\sE$-exceptional but $\fm$-negligible set $N\subset E$,  the same (semi-)Dirichlet form on $L^2(E\setminus N,\fm)$ is not quasi-regular.  Below, we provide an example illustrating the existence of a non-tight right process associated with such a non-quasi-regular (semi-)Dirichlet form.

\begin{example}\label{EXA45}
Let $Z'\subset \bR^n$ with $n\geq 4$ be the universally measurable set in Example~\ref{EXA34}.  Then the restriction of Brownian resolvent to $\bR^n\setminus Z'$ corresponds to a right process $\tilde{X}$.  Clearly it is symmetric with respect to the Lebesgue measure $\fm$ and associated with the Dirichlet form $(\frac{1}{2}\mathbf{D}, H^1(\bR^n))$ on $L^2(\bR^n\setminus Z')$,  where $H^1(\bR^n)$ is the Sobolev space of order $1$ over $\bR^n$ and $\mathbf{D}(f,g):=\int_{\bR^n}\nabla f(x) \nabla g(x)dx$.  However,  any Borel set containing $Z'$ is not polar.  Hence Theorem~\ref{THM44} tells us that $(\frac{1}{2}\mathbf{D}, H^1(\bR^n))$  is not quasi-regular on $L^2(\bR^n\setminus Z')$.  In view of Theorem~\ref{THM43},  $\tilde{X}$ is not $\fm$-tight.  
\end{example}


The choice of the inherited topology in Theorem~\ref{THM44} is indeed crucial. The derivation of formula \eqref{eq:41} relies on the specific setting of the inherited topology.  In contrast, \eqref{eq:41} appears as a fundamental assumption in Theorems \ref{THM31} and \ref{THM32}.  If different topologies are assigned to $\tilde{E}$, then even if $N$ is a non-$\sE$-exceptional set, the semi-Dirichlet form $(\sE,\sF)$ may exhibit quasi-regularity on $L^2(\tilde{E},\fm)$. An interesting example demonstrating this phenomenon is presented in \cite{BOZ06} (specifically, the example before Proposition~3.5). Additionally, the following example, due to \cite{Sc79, LS20}, further illustrates this point.


\begin{example}\label{EXA46}
Let $\mathbb G:=(-\infty,0_-]\cup [0,\infty)$ consist of two separate intervals,  where $0_-$ is viewed as a distinct zero point from $0$.  The Lebesgue measure on $\mathbb G$ is denoted by $\fm$.  Consider the Dirichlet form 
\[
\begin{aligned}
	&\sF:=\left\{f\in L^2(\mathbb{G}): f|_{(0,\infty)}\in H^1((0,\infty)), f|_{(-\infty, 0_-)}\in H^1((-\infty,  0_-))\right\},\\
	&\sE(f,f):=\frac{1}{2}\int_{\mathbb{G}\setminus \{0_-,0\}} f'(x)^2dx+\frac{\kappa}{4}(f(0)-f(0_-))^2,\quad f\in \sF,
\end{aligned}\]
where $H^1((0,\infty))$ (resp.  $H^1((-\infty,  0_-))$) is the Sobolev space of order $1$ over $(0,\infty)$ (resp.  $(-\infty,  0_-)$),  $f(0)$ (resp.  $f(0_-)$ is the right (resp. left) limit of $f$ at $0$ (resp. $0_-$),  and $\kappa>0$ is a given constant.  Note that $f\in \sF$ is  tacitly continuous on $(0,\infty)$ and $(-\infty, 0_-)$ respectively.  This Dirichlet form is regular (hence quasi-regular) on $L^2(\mathbb{G})$,  and its associated Markov process $X$ is indeed a Feller process on $\mathbb{G}$,  called the \emph{snapping-out Brownian motion} with parameter $\kappa$ (see \cite{L16}).   In addition,  the transition semigoup $P_t$ as well as its resolvent $U^\alpha$ of $X$ is absolutely continuous with respect to $\fm$,  and  every singleton set is not $\sE$-polar.  We refer readers to \cite{LS20} for more details.  

Obviously $\{0_-\}$ is $\fm$-negligible and hence $U$-negligible,  while $\tilde{E}:=\mathbb{G}\setminus \{0-\}$ is not (quasi-)absorbing.  
We will equip $\tilde{E}$ with three kinds of topologies generating the same Borel $\sigma$-algebra identical to $\mathcal{B}(\bR)$ and,  relative to which examine the Dirichlet form $(\sE,\sF)$ on $L^2(\tilde{E},\fm)$.  

The first topology on $\tilde{E}$ is the inherited topology of $\mathbb{G}$.  In this case the restriction $(\tilde{P}_t)$ of $(P_t)$ to $\tilde E$ is still a transition function acting on $L^2(\tilde{E},\fm)$ as a strongly continuous symmetric contraction semigroup; see \cite{Ch22}.  Particularly,  $(\sE,\sF)$ is the Dirichlet form of $(\tilde{P}_t)$ on $L^2(\tilde{E},\fm)$.  Note that $(\tilde{P}_t)$ does not satisfy (HD1) due to Theorem~\ref{THM24},  and $(\sE,\sF)$ is not quasi-regular on $L^2(\tilde{E},\fm)$.  

Next we endow $\tilde{E}$ with the Euclidean topology,  so that it is identical to $\bR$.  Currently the restricted transition function $(\tilde{P}_t)$ on $\bR$ satisfies (HD1),  and one of its realizations can be obtained by mapping $X$ to $\bR$ through  the surjection $\mathbf{r}:\bG\rightarrow \bR$ with $\mathbf{r}(x)=x$ for $x\notin \{0,0_-\}$ and $\mathbf{r}(0)=\mathbf{r}(0_-)=0$;  see \cite[Theorem~4.1]{Sc79}.  The resulting process $\tilde{X}$ is continuous and symmetric,  but does not satisfy the strong Markov property,  i.e.  (HD2) fails for $(\tilde{P}_t)$.  The Dirichlet form of $(\tilde{P}_t)$ on $L^2(\bR)$ is still $(\sE,\sF)$,  which is obviously not quasi-regular.  

Finally we construct a third topology $\tilde \tau$ on $\tilde{E}$ so that $(\sE,\sF)$ is quasi-regular on $L^2(\tilde{E},\fm)$ relative to this topology by mimicking the argument in \cite{BOZ06}.  To do this,  let $\tilde{M}:=\{-1/n: n\geq 1\}$ and consider the bijection $\varphi: \bG\rightarrow \tilde{E}$ defined as $\varphi(x):=x$ for $x\in \tilde{E}\setminus \tilde{M}$,  $\varphi(0_-):=-1$ and $\varphi(-1/n):=-1/(n+1)$ for $n\geq 1$.  Let $\tau$ be the original topology of $\bG$.  Set $\tilde{\tau}:=\varphi(\tau)=\{\varphi(O): O\in \tau\}$,  the biggest topology on $\tilde{E}$ making $\varphi$ continuous.  It is easy to see that $\tilde{\tau}$ is Haursdorff and hence $\varphi: (\bG,  \tau)\rightarrow (\tilde{E},  \tilde{\tau})$ is a homoemorphism.  Every $f\in \sF$ admits a $\tilde{\tau}$-continuous version $\tilde{f}$ with $\tilde{f}(x):=f(x)$ for $x\in \tilde{E}\setminus \tilde{M}$,  $\tilde{f}(-1):=f(0_-)$ and $\tilde{f}(-1/(n+1)):=f(-1/n)$ for $n\geq 1$,  and hence 
\[
	\sE(f,f)=\frac{1}{2}\int_{\tilde{E}\setminus \tilde{M}}\tilde{f}'(x)^2dx+\frac{\kappa}{4}(\tilde{f}(0)-\tilde{f}(-1))^2.  
\]
Obviously $(\sE,\sF)$ is regular (thus quasi-regular) on $L^2(\tilde{E},\fm)$ relative to the topology $\tilde{\tau}$,  whose associated Markov process is the homeomorphic image of $X$ under $\varphi$.  Denote by $(\tilde{P}^{\tilde{\tau}}_t)$ its associated transition semigroup.  Then $\tilde{P}^{\tilde{\tau}}_t$ is identical with $P_t(\varphi^{-1}(\cdot),  \varphi^{-1}(\cdot))$.  Particularly,  it is not the same as the restriction of $(P_t)$ to $\tilde{E}$.  
\end{example}

\appendix

\section{Basics of right processes}\label{APPA}

\subsection{Radon space}
Let $E$ be a Radon topological space,  i.e.  it is homeomorphic to a universally measurable subset of a compact metric space.  The Borel $\sigma$-algebra on $E$ is denoted by  $\mathcal{B}(E)$,  and the universal completion of $\mathcal{B}(E)$ is denoted by $\mathcal{B}^u(E)$,  i.e. 
\[
	\cB^u(E)=\cap_{\mu} \overline{\cB(E)}^\mu,
\] 
where $\overline{\cB(E)}^\mu$ is the completion of $\cB(E)$ with respect to $\mu$ running over all finite positive measures on $(E,\mathcal{B}(E))$.  Note that every finite measure on $(E,\mathcal{B}(E))$ extends in a unique way to a measure on $(E,\mathcal{B}^u(E))$,  and every finite measure on $(E,\mathcal{B}^u(E))$ is the unique extension of its restriction to $\cB(E)$; see,  e.g.,  \cite[(A1.1)]{Sh88}.  Hence we would also call $\mu$ just a measure on $E$ if no confusions caused.  The simpler notations $\cE$ and $\cE^u$ in place of $\cB(E)$ and $\cB^u(E)$ will be used unless clarity dictates otherwise.  

Suppose $E$ is embedded into the compact metric space $\hat{E}$ with the metric $d$ compatible to the topology on $E$,  i.e.  the subspace topology of $E$ relative to $(\hat{E}, d)$ is identical to the original topology of $E$.  Note that $\cE=\cB(\hat{E})|_E:=\{E\cap \hat B: \hat{B}\in \cB(\hat{E})\}$ and $\cE^u=\cB^u(\hat{E})|_E:=\{E\cap \hat B: \hat{B}\in \cB^u(\hat{E})\}$; see,  e.g.,  \cite[(A2.2)]{Sh88}.
 We set $C(E)=C(E,d)$ to be the family of all real continuous functions on $E$ and $C_d(E)$ to be the family of all real $d$-uniformly continuous functions on $E$.  Then $C_d(E)$ is a separable space with respect to the uniform norm while $C(E)$ may be not.

\subsection{Filtered measurable space}

Given a $\sigma$-algebra $\mathcal{M}$ on a space $M$,  $b\mathcal{M}$ (resp.  $p\mathcal{M})$) stands for the class of bounded (resp.  $[0,\infty]$-valued) $\mathcal{M}$-measurable functions on $M$.  Given two measurable spaces $(M_1,\cM_1)$ and $(M_2,\cM_2)$.  A map $f: M_1\rightarrow M_2$ is measurable,  denoted by $f\in \cM_1/\cM_2$, if $f^{-1}(B_2)\in \cM_1$ for any $B_2\in \cM_2$.  

Let $(\Omega, \cG, \bP)$ be a probability space and $(X_t)_{t\geq 0}$ be a stochastic process with values in $E$.  That is,  $X_t$,  $t\geq 0$,  is a collection of measurable maps from $(\Omega,  \cG)$ to $(E,\cE)$.  To emphasize the dependence on $\cE$,  we call it an $\cE$-stochastic process.  Similar definitions will apply when $\cE$ is replaced by larger $\sigma$-algebra $\cE^\bullet$,  e.g.,  $\cE^u$.  In the current paper a stochastic process tacitly means an $\cE^u$-stochastic process.  It is required in this case that for any $t\geq 0$,  $\{X_t\in B\}:=\{\omega\in \Omega: X_t(\omega)\in B\}\in \cG$ for every $B\in \cE^u$ rather than every $B\in \cE$.  A filtration $(\cG_t)$ means an increasing family of sub-$\sigma$-algebras of $\cG$,  to which $(X_t)$ is ($\cE^u$-)adapted in the sense that $X_t\in \cG_t/\cE^u$ for $t\geq 0$.  Then the collection $(\Omega,  \cG_t,\cG)$ is called a filtered measurable space.

Corresponding to a fixed $\cE^u$-stochastic process $(X_t)$ on $\Omega$,  the natural $\sigma$-algebra $\cF^u_t$ is defined as $\sigma\left\{f(X_s): 0\leq s\leq t, f\in \cE^u\right\}$ and $\cF^u:=\sigma\left\{f(X_t): t\geq 0, f\in \cE^u\right\}$.  Obviously $\cF^u_t\subset \cG_t$ and $\cF^u\subset \cG$. 

\subsection{Transition function}

Fix a $\sigma$-algebra $\cE^\bullet$ on $E$ with  $\cE\subset \cE^\bullet \subset \cE^u$. 
The notation $K(x,dy)$ is a kernel on $(E,\cE^\bullet)$ provided that,  for all $x\in E$,  $K(x,dy)$ is a positive measure on $(E,\cE^\bullet)$,  and for every $B\in \cE^\bullet$,  $x\mapsto K(x,B)$ is $\cE^\bullet$-measurable.  
It is called a Markov  (sub-Markov) kernel if $K(x,E)=1$ (resp.  $K(x,E)\leq 1$) for all $x\in E$.  Note that $K$ can be always extended to a kernel on $(E,\cE^u)$,  which is still denoted by $K$.  

\begin{definition}\label{DEFA1}
A family $(P_t)_{t\geq 0}$ of Markov kernels on $(E,\cE^\bullet)$ is called a \emph{transition function} on $(E,\cE^\bullet)$ if for all $t,s\geq 0$ and all $x\in E,  B\in \cE^\bullet$,
\begin{equation}\label{eq:A1}
	P_{t+s}(x,B)=\int_E P_t(x,dy)P_s(y,B).
\end{equation}
It is called a \emph{normal} transition function if in addition,  $P_0(x, B)=1_B(x)$.  
\end{definition}

Let $(X_t)$ be  a stochastic process on $(\Omega, \cG,\bP)$ $\cE^\bullet$-adapted to $(\cG_t)$.  It satisfies the \emph{Markov property} relative to a transition semigroup $(P_t)$ on $(E,\cE^\bullet)$ provided that 
\begin{equation}\label{eq:A2}
	\bP\{f(X_{t+s})|\cG_t\}=P_sf(X_t),\quad t,s\geq 0,  f\in b\cE^\bullet.  
\end{equation}
The transition function $(P_t)$ is also called the (Markov) transition semigroup of $(X_t)$ due to the semigroup property \eqref{eq:A1}.  
The distribution $\mu$ of $X_0$ is called the \emph{initial law} of $(X_t)$.  
Clearly $\mu_t=\mu P_t$ is the distribution of $X_t$.  

\subsection{The first regularity hypothesis (HD1)}

Let $(X_t)_{t\geq 0}$ be a stochastic process defined on $(\Omega, \cG,\bP)$ and having valued in a topological space $E$.  It is called \emph{right continuous} in case that every sample path $t\mapsto X_t(\omega)$ is a right continuous map from $\bR^+:=[0,\infty)$ to $E$.  The following hypothesis,  due to \cite[(2.1)]{Sh88},  is essentially the first of Meyer's \emph{hypoth\`eses droites}.  

\begin{definition}[HD1]\label{DEFA2}
Let $E$ be a Radon topological space. 
A Markov transition function $(P_t)_{t\geq 0}$ on $(E, \cE^u)$ is said to satisfy (HD1),  if given an arbitrary probability law $\mu$ on $E$,  there exists an $E$-valued right continuous stochastic process $(X_t)_{t\geq 0}$ on some filtered probability space $(\Omega, \cG,\cG_t,  \bP)$ so that $X=(\Omega, \cG,\cG_t,\bP,  X_t)$ satisfies the Markov property \eqref{eq:A2} ($\cE^\bullet=\cE^u$) with transition semigroup $(P_t)_{t\geq 0}$  and initial law $\mu$. 
\end{definition}

In order to facilitate computations we shall work with a fixed collection of random variables $X_t$ defined on some probability space,  and a collection $\bP^x$ specified in such a way that $\bP^x(X_0=x)=1$, and under every $\bP^x$,  $X_t$ is Markov with semigroup $(P_t)$.  That is the following.

\begin{definition}\label{DEFA3}
Let $E$ be a Radon space and $(P_t)_{t\geq 0}$ be a Markov transition function satisfying (HD1).  The collection $X=(\Omega, \cG,\cG_t, X_t,\theta_t, \bP^x)$ is called a \emph{realization} of $(P_t)$ if it satisfies the following conditions:
\begin{itemize}
\item[(1)] $(\Omega, \cG,\cG_t)$ is a filtered measurable space,  and $X_t$ is an $E$-valued right continuous process $\cE^u$-adapted to $(\cG_t)$;
\item[(2)] $(\theta_t)_{t\geq 0}$ is a collection of shift operators mapping $\Omega$ into itself and satisfying for $t,s\geq 0$, $\theta_t\circ \theta_s=\theta_{t+s}$ and $X_t\circ \theta_s=X_{t+s}$;
\item[(3)] For every $x\in E$,  $\bP^x(X_0=x)=1$,  and the process $(X_t)_{t\geq 0}$ has the Markov property \eqref{eq:A2} with transition semigroup $(P_t)$ relative to $(\Omega,\cG,\cG_t, \bP^x)$.   
\end{itemize}
Furthermore,  a realization of $(P_t)$ is called \emph{canonical} if $\Omega$ is the space of all right continuous maps from $\bR^+$ to $E$,  $X_t(\omega):=\omega(t)$,  $\cG=\sigma\left\{f(X_t): f\in \cE^u,t\geq 0\right\}$ and $\cG_t=\sigma\left\{f(X_s): f\in \cE^u,0\leq s\leq t\right\}$.
\end{definition}
\begin{remark}\label{RMA4}
\begin{itemize}
\item[(1)] The normal property $\bP^x(X_0=x)=1$ is not always built into the definition of Markov property \eqref{eq:A2}.  Particularly,  it implies that the transition semigroup $(P_t)$ is normal. 
\item[(2)] Obviously $\theta_t\in \cF^u_{t+s}/\cF^u_s$ for any $s\geq 0$.  
\item[(3)] The existence of realization is equivalent to (HD1) for a normal transition function $(P_t)$.  One one hand,  such a (canonical) realization of a normal $(P_t)$ satisfying (HD1) always exists.   More precisely,  define $\theta_t\omega(s):=\omega(t+s)$ for $t,s\geq 0$ and $\omega\in \Omega$,  the space of all right continuous maps from $\bR^+$ to $E$.  For $x\in E$,  let $(\tilde\Omega,  \cG,\cG_t, \bP, \tilde X_t)$ be the collection in Definition~\ref{DEFA2} with $\mu=\delta_x$.   Define a map $\Phi: \tilde\Omega\rightarrow \Omega$,  $\tilde\omega\mapsto \Phi(\tilde\omega)$ with $\Phi(\tilde \omega)(t):=\tilde X_t(\tilde\omega)$,  as is characterized by the formulae $\tilde X_t=X_t\circ \Phi$,  $t\geq 0$.  Obviously $\Phi\in \cG/\cF^u$.  Let $\bP^x$ be the image measure of $\bP$ under the map $\Phi$,  so that $X=(\Omega, \cF^u,\cF^u_t, X_t,\theta_t, \bP^x)$ is the canonical realization of $(P_t)$. 
To the contrary,  let $X=(\Omega, \cG,\cG_t, X_t,\theta_t, \bP^x)$ be a realization of $(P_t)_{t\geq 0}$.  Note that $x\mapsto \bP^x (B)$ is $\cE^u$-measurable for every $B\in \cF^u$;  see \cite[(2.6)]{Sh88}.  Hence $\bP^\mu(\cdot):=\int_E \mu(dx)\bP^x(\cdot)$ defines a probability measure on $(\Omega,\cF^u)$ for any probability measure $\mu$ on $E$,  and $(X_t)$ satisfies the Markov property relative to $(\Omega, \cF^u,\cF^u_t,\bP^\mu)$,  with transition semigroup $(P_t)$ and initial law $\mu$.  In other words,  $(P_t)$ satisfies (HD1).  
\end{itemize}
\end{remark}

\subsection{Augmented filtration}\label{APPA5}

Now we introduce the notation of augmentation of the natural filtration $\cF^u_t$ (not necessarily on the space of right continuous maps).  The augmentation of general filtration is analogous.  

Given an initial law $\mu$ on $E$,  let $\cF^\mu$ denote the completion $\cF^u$ relative to $\bP^\mu$,  and let $\cN^\mu$ denote the $\sigma$-ideal of $\bP^\mu$-null sets in $\cF^\mu$.  Define $\cF:=\cap_\mu \cF^\mu$, where $\mu$ runs over all initial laws on $E$,  $\cN:=\cap_\mu \cN^\mu$,  $\cF^\mu_t:=\cF^u_t \vee \cN^\mu$ and  $\cF_t:=\cap_\mu \cF^\mu_t$.
Two random variables $G,H\in \cF$ are called \emph{a.s.  equal} if $\{G\neq H\}\in \cN$.  The filtration $(\cF_t)$ is called the \emph{augmented natural filtration} on $\Omega$.  

Let $X=(\Omega,\cG,\cG_t,X_t,\theta_t, \bP^x)$ be a realization of $(P_t)$ in Definition~\ref{DEFA3}.  Then $\theta_t\in \cF_{t+s}/\cF_{s}$ for any $s\geq 0$,  and $(\Omega, \cF,\cF_t,X_t,\theta_t,\bP^x)$ is also a realization of $(P_t)$.  

\subsection{The second fundamental hypothesis}

We assume throughout this subsection that $X=(\Omega,\cG,\cG_t,  X_t, \theta_t, \bP^x)$ is a right continuous Markov process with transition semigroup $(P_t)$ on a Radon space $E$.  

The resolvent $(U^\alpha)_{\alpha\geq 0}$ is the family of kernels on $(E,\cE^u)$ defined by
\[
	U^\alpha f(x):=\bP^x \int_0^\infty e^{-\alpha t}f(X_t)dt,\quad \alpha\geq 0,  f\in p \cE^u.
\]
It satisfies the well-known resolvent equation: For $0\leq \alpha\leq \beta$ and $f\in p\cE^u$,
\[
	U^\alpha f=U^\beta f+(\beta-\alpha) U^\alpha U^\beta f.  
\]
The family of $\alpha$-excessive functions is crucial in the theory of Markov processes.

\begin{definition}\label{DEFA5}
Let $\alpha\geq 0$ and $f\in p\cE^u$.  Then $f$ is \emph{$\alpha$-super-mean-valued} in case $e^{-\alpha t}P_t f\leq f$ for all $t\geq 0$,  and $f$ is \emph{$\alpha$-excessive} if in addition,  $e^{-\alpha t}P_t f\rightarrow f$ as $t\downarrow 0$.  It is called simply excessive if $f$ is $0$-excessive.  The classes of $\alpha$-excessive,  excessive functions are denoted by $\cS^\alpha$,  $\cS$ respectively.  
\end{definition}
\begin{remark}\label{RMA6}
A function $f\in p\cE^u$ is called \emph{$\alpha$-supermedian} in case $\beta U^{\alpha+\beta} f\leq f$ for all $\beta>0$.  Note that $\alpha$-super-mean-valued functions are $\alpha$-supermedian,  but not vice versa.  In addition,  $\cS^\alpha=\{f \text{ is }\alpha\text{-supermedian}: \lim_{\beta\uparrow \infty}\beta U^{\alpha+\beta}f=f\}$.  
\end{remark}

We take up now the second fundamental hypothesis,  which makes the strong Markov property available.

\begin{definition}[HD2]\label{DEFA7}
The Markov process $X=(\Omega,\cG,\cG_t,  X_t, \theta_t, \bP^x)$ with transition semigroup $(P_t)$ is said to satisfy (HD2),  if for every $\alpha>0$ and every $f\in \cS^\alpha$, the process $t\mapsto f(X_t)$ is a.s.  right process.
\end{definition}
\begin{remark}
In view of \eqref{APPA5},  ``\emph{a.s.}" means that there is $N\in \cN(\cG)$ such that $t\mapsto f(X_t)$ is right continuous on $\Omega\setminus N$,  where $\cN(\cG)$ is the intersection of all $\sigma$-ideal of $\bP^\mu$-null sets in the completion of $\cG$ relative to $\bP^\mu$.  Briefly speaking,  it says that for any $x\in E$,  $t\mapsto f(X_t)$ is $\bP^x$-a.s.  right continuous.  This hypothesis always implies the strong Markov property of $X$ in the sense of \cite[\S6]{Sh88}.  Particularly,  if $E$ is Lusin,  i.e.  it is homeomorphic to a Borel subset of a compact metric space,  and $P_t(b\cE)\subset b\cE$,  then (HD2) is equivalent to the strong Markov property of $X$. 
\end{remark}

\subsection{Right processes and right semigroups}

Now we have a position to raise the formal definition of right processes.  

\begin{definition}\label{DEFA9}
A system $X=(\Omega, \cG,\cG_t,X_t,\theta_t,\bP^x)$ is a \emph{right process} on the Radon space $E$ with transition semigroup $(P_t)$ provided:
\begin{itemize}
\item[(i)] $X$ is a realization of $(P_t)$;
\item[(ii)] $X$ satisfies (HD2);
\item[(iii)] $(\cG_t)$ is augmented and right continuous.  
\end{itemize}
If there is some right process with transition semigroup $(P_t)$,  then $(P_t)$ is called a \emph{right semigroup}.  
\end{definition}
\begin{remark}
\begin{itemize}
\item[(1)] When $E$ is a Lusin topological space and $P_t(b\cE)\subset b\cE$,  $X$ is called a Borel right process.  
\end{itemize}
\item[(2)] If $X$ is a right process,  then the augmented natural filtration $\cF_t$ is right continuous and $(\Omega, \cF,\cF_t, X_t, \theta_t,\bP^x)$ is also a right process with transition semigroup $(P_t)$.  Particularly,   if $(P_t)$ is a right semigroup,  then its \emph{augmented canonical realization}, obtained by replacing the natural filtration in the canonical realization with the augmented one,  is a right process with transition semigroup $(P_t)$. 
\end{remark}

Let $X$ be a right process on $E$ with right semigroup $(P_t)$.  
The fine topology (see \cite[\S10]{Sh88}) is the coarsest topology on $E$ making all functions in $\cup_{\alpha> 0}\cS^\alpha$ continuous.  It is finer than the original topology of $E$. The Borel $\sigma$-algebra relative to the fine topology on $E$ is denoted by $\cE^e$.  Actually $\cE\subset \cE^e=\sigma\{\cup_{\alpha>0} \cS^\alpha\}$ and $P_t(b\cE^e)\subset b\cE^e$.  For any $B\in \cE^e$ (more generally,  if $B$ is nearly optional in the sense of \cite[(5.1)]{Sh88}),  the first hitting time $T_B:=\inf\{t>0: X_t\in B\}$ is an $\cF_{t}$-stopping time,  i.e.  $\{T_B\leq t\}\in \cF_t$ for any $t\geq 0$.  

\subsection{Lifetime}\label{APPA8}
If a transition function $(P_t)$ is only sub-Markovian,  it may be extended to a Markov one $(\tilde{P}_t)$ on a larger space $E_\Delta$ by a standard argument as in \cite[(11.1)]{Sh88}. (Take an abstract point $\Delta$ not in $E$ and let $E_\Delta:=E\cup \{\Delta\}$ be the Radon space obtained by adjoining $\Delta$ to $E$ as an isolated point.) In this case we call $(P_t)$ a right semigroup if $(\tilde{P}_t)$ is a right semigroup.  

Realizing the right semigroup $(\tilde{P}_t)$ as a right process $(\Omega, \tilde{\cG},\tilde{\cG}_t,\tilde{X}_t, \tilde{\theta}_t,\tilde{\bP}^x)$ on $E_\Delta$,  one gets  $\tilde{\bP}^\Delta(\tilde{X}_t=\Delta, \forall t\geq 0)=1$.  Let $\zeta:=\{t>0:\tilde{X}_t=\Delta\}$. By the strong Markov property,  $\tilde{X}_t=\Delta$ for all $t>\zeta$,  almost surely.  Hence $\Delta$ is usually called the \emph{ceremony} for the process and $\zeta$ is called the \emph{lifetime}.  The role played by $\tilde{P}_t$ is de-emphasized by making the convention that every function on $E$ is automatically extended to $E_\Delta$ by setting $f(\Delta):=0$.  Defining $X_t$ to be the same as $\tilde{X}_t$ on $\Omega$ and letting $\bP^x:=\tilde{\bP}^x$ for $x\in E$,   we can obtain another collection $(\Omega, \cG,\cG_t,X_t,\theta_t, \bP^x)$ on $E$ in certain standard manner,  which is called the right process on $E$ with lifetime $\zeta$ and transition semigroup $(P_t)$.  More details are referred to in \cite[\S11]{Sh88}.


\section{Ray-Knight compactification}\label{APPB}

Let $X=(\Omega, \cG,\cG_t,X_t,\theta_t,\bP^x)$ be a right process on a Radon space $E$ with right semigroup $(P_t)$.  The resolvent of $X$ is denoted by $(U^\alpha)_{\alpha\geq 0}$.  If $X$ has the ceremony,  it should be regarded as a point in $E$ throughout this section.  The following introduction to Ray-Knight compactification is due to \cite[\S9, \S17,  \S18,  \S39]{Sh88} and \cite{Ge75}. 

Let $\mathbb{Q}$ denote the set of rational numbers,  $\bQ^+$ the positive rational numbers and $\bQ^{++}$ the strictly positive rational numbers.  The $\bQ^+$-cone generated by a family $\cY$ of positive and bounded functions on $E$ is the set of all $\bQ^+$-linear combinations of functions in the class $\cY$.  Given a $\bQ^+$-cone $\cY\subset bp\cE^u$,  set 
\[
\begin{aligned}
 &\bigwedge(\cY):=\{k_1\wedge \cdots \wedge k_n: n\geq 1, k_1,\cdots,  k_n\in \cY\}, \\
 &\cU(\cY):=\{U^{\alpha_1}k_1+\cdots+U^{\alpha_n}k_n: n\geq 1,\alpha_i\in \bQ^{++},k_i\in\cY\}.
 \end{aligned}\]
Both operations of $\bigwedge$ and $\cU$ keep the property of $\bQ^+$-cone. 

Recall that $C_d(E)$ is the family of all $d$-uniformly continuous functions on $E$.  For convenience,  we propose to assign a name to the following class of functions.

\begin{definition}\label{DEFB1}
A family $\cC$ is called a \emph{pre-Ray class},  if
\begin{itemize}
\item[(i)] $\cC\subset p C_d(E)$ is countable;
\item[(ii)] $1_E\in \cC$;
\item[(iii)] The linear span of $\cC$ is uniformly dense in $C_d(E)$.  
\end{itemize}
\end{definition}

Since $C_d(E)$ is separable,  such $\cC$ always exists.  The \emph{rational Ray cone} $\cR$ generated by $(U^\alpha)$ and $\cC$ is the $\bQ^+$-cone defined as follows: Let $\cH$ denote the $\bQ^+$-cone generated by $\cC$, and let $\cR_0:=\mathcal{U}(\cH)$.  For $n\geq 1$,  let $\cR_n:=\bigwedge(\cR_{n-1}+\cU(\cR_{n-1}))$,  and finally set $\cR:=\cup_{n\geq 0}\cR_n$.  Obviously $\cR\subset \cup_{\alpha>0} b \cS^\alpha$,  and $\cR$ is countable,  stable under the operation $\bigwedge$,  contains the positive rational constant functions, and separates the points of $E$.  

Write $\cR=\{g_n:n\geq 1\}$.  Define a metric $\rho$ on $E$ as
\[
	\rho(x,y):=\sum_{n\geq 1}2^{-n}\|g_n\|^{-1}|g_n(x)-g_n(y)|,
\]
where $\|g_n\|:=\sup_{x\in E}|g_n(x)|$.  The map 
\[
	\Psi: E\rightarrow K:=\prod_{n=1}^\infty [0, \|g_n\|],  \quad x\mapsto (g_n(x))_{n\geq 1}
\]
is an injection.  Since the product topology on $K$ is generated by the metric $$\rho'(a,b):=\sum_{n\geq 1}2^{-n}\|g_n\|^{-1}|a_n-b_n|$$ for $a=(a_n)_{n\geq 1}$ and $b=(b_n)_{n\geq 1}$,  $\Psi$ is an isometry of $(E,\rho)$ to $(K,\rho')$.  It follows that the completion $(\bar{E},\bar{\rho})$ of $(E,\rho)$ is compact.  In general $\Psi$ is only $\cE^u$-measurable and $\cE^u=\mathcal{B}^u(\bar{E})|_E$.  If $X$ is a Borel right process,  then $\Psi$ is $\cE$-measurable. 

The topology on $E$ induced by the metric $\rho$ is called the \emph{Ray topology} on $E$.  Actually it does not depend on the choice of $d$ or $\cC$,  and in general,  is not compatible to the original topology on $E$.  Denote by $C_\rho(E)$ the space of $\rho$-uniformly continuous functions on $E$.  Then for all $\alpha>0$,  $U^\alpha(C_\rho(E))\subset C_\rho(E)$,  $U^\alpha(C_d(E))\subset C_\rho(E)$ and $\cR-\cR$ is uniformly dense in $C_\rho(E)$; see \cite[(17.8)]{Sh88}. 
 Let $\cE^r:=\sigma\{C_\rho(E)\}$,  i.e.  the $\sigma$-algebra on $E$ generated by the Ray topology.  Then $\cE\subset \cE^r\subset \cE^u$ and $P_t(b\cE^r)\subset b\cE^r$,  $U^\alpha(b\cE^r)\subset b\cE^r$ for $\alpha>0$.  

We may now construct by continuity a resolvent $\bar{U}^\alpha$ on $\bar{E}$ that extends $U^\alpha$ on $E$.  Let $\bar{f}\in C(\bar{E})$ be the continuous extension of $f\in C_\rho(E)$.  Then $U^\alpha f\in C_\rho(E)$ and so $U^\alpha f$ extends continuously to $\overline{U^\alpha f}\in C(\bar{E})$.  Define the map $\bar{U}^\alpha: C(\bar{E})\rightarrow C(\bar{E})$ as
\[
	\bar{U}^\alpha \bar{f}:=\overline{U^\alpha f},\quad \bar{f}\in C(\bar{E}).  
\]
Then $(\bar{U}^\alpha)_{\alpha>0}$ is a so-called \emph{Ray resolvent} on $C(\bar{E})$ in the sense of,  e.g.,  \cite[(9.4)]{Sh88}.  The collection $(\bar{E},\bar{\rho},\bar{U}^\alpha)$ is called the \emph{Ray-Knight compactification} (or \emph{Ray-Knight completion}) of $(E,d,U^\alpha)$.  It depends not only on $E$,  $d$ and $U^\alpha$ but also on the choice of $\cC$.  
For every $x\in E$,  $\bar{U}^\alpha(x,\cdot)$ is carried by $E\in \mathcal{B}^u(\bar{E})$ and its restriction to $E$ is equal to $U^\alpha(x,\cdot)$.  Let $\bar P_t$ be the \emph{Ray semigroup} associated with $\bar{U}^\alpha$ on $\bar{E}$.  Then for all $x\in E$ and $t\geq 0$,  $\bar{P}_t(x,\cdot)$ is also carried by $E$ and its restriction to $E$ is equal to $P_t(x,\cdot)$.  

The \emph{Ray process} $\bar{X}$ associated with $\bar{U}^\alpha$ on $\bar{E}$ admits branching points $B:=\{x\in \bar E: \bar{P}_0(x,\cdot)\neq \delta_x\}$ but may lead to certain right processes by restriction.  Note that $B\in \mathcal{B}(\bar{E})$ and the set of non-branching points $D:=\bar{E}\setminus B$ always contains $E$.  The (first) restriction of $\bar{X}$ to $D$ ($\supset E$) is a Borel right process.  On the other hand,  put
\[
	E_R:=\{x\in \bar{E}: \bar{U}^\alpha(x,\cdot)\text{ is carried by }E\},
\]
which inherits the subspace topology from $\bar{E}$.  Then $E_R$ is  independent of $\alpha$, $E\subset E_R$, and $E_R$ is  a Radon topological space.  This space,  called the \emph{Ray space},  is also independent of $d$ and $\cC$,  and depends only on $U^\alpha$ and the original topology on $E$.  Furthermore,  the (second) restriction of $\bar{X}$ to $E_D:=E_R\cap D$ ($\supset E$) is also a (not necessarily Borel) right process.



\bibliographystyle{alpha} 
\bibliography{ResRight}

\end{document}